\documentclass{amsart}

\usepackage{amssymb}                     % Matematiske symboler
\usepackage{amsmath}                     % Matematiske kommandoer
\usepackage{amsfonts}
\usepackage{amsthm}                      % Sætning, Korollar osv.
\usepackage{xspace}
\usepackage{tabularx}                    % Dejlige tabeller
\usepackage[all]{xy}
\CompileMatrices
\usepackage{rotating}                    % Rotér elementer

\numberwithin{equation}{subsection}

\renewcommand{\H}{\operatorname{H}}
\newcommand{\tor}{\operatorname{Tor}}
\newcommand{\ext}{\operatorname{Ext}}
\renewcommand{\hom}{\operatorname{Hom}}

\newcommand{\length}{\operatorname{length}}

\newcommand{\spec}{\operatorname{Spec}}
\newcommand{\supp}{\operatorname{Supp}}

\newcommand{\var}{V}
\newcommand{\im}{\operatorname{Im}}
\renewcommand{\ker}{\operatorname{Ker}}

\newcommand{\codim}{\operatorname{codim}}

\newcommand{\vandim}{\operatorname{vdim}}
\newcommand{\f}[1]{{^{\mathit{f}^{#1}}}\!}
\newcommand{\frob}{F}
\newcommand{\herzog}{G}
\newcommand{\grot}{\mathbb{G}}

\newcommand{\proj}{\mathsf{P^\mathrm{f}}}
\newcommand{\inj}{\mathsf{I^\mathrm{f}}}
\newcommand{\derg}{\mathsf{D}}
\newcommand{\projg}{\mathsf{P}}
\newcommand{\fla}{\mathsf{F^\mathrm{f}}}

\newcommand{\injg}{\mathsf{I}}

\newcommand{\der}{\mathsf{D}^{\mathrm{f}}_\Box}
\newcommand{\ltensor}{\otimes^\mathbf{L}}
\newcommand{\rhom}{\operatorname{\mathbf{R}Hom}}
\newcommand{\lfrob}{\mathbf{L}F}
\newcommand{\rherzog}{\mathbf{R}G}

\newcommand{\id}{\operatorname{id}}

\newcommand{\XXX}{\mathfrak{X}}
\newcommand{\YYY}{\mathfrak{Y}}

\newcommand{\ppp}{\mathfrak{p}}
\newcommand{\qqq}{\mathfrak{q}}
\newcommand{\aaa}{\mathfrak{a}}
\newcommand{\mmm}{\mathfrak{m}}
\newcommand{\NN}{\mathbb{N}}
\newcommand{\ZZ}{\mathbb{Z}}
\newcommand{\QQ}{\mathbb{Q}}

\newcommand{\comp}[1]{{#1}^c}
\newcommand{\compcomp}[1]{{#1}^{cc}}
\newcommand{\shift}{\mathsf{\Sigma}}
\newcommand{\toeq}{\overset{\simeq}{\longrightarrow}}
\newcommand{\oteq}{\overset{\simeq}{\longleftarrow}}

\renewcommand{\epsilon}{\varepsilon}
\renewcommand{\phi}{\varphi}
\renewcommand{\tilde}{\widetilde}
\renewcommand{\leq}{\leqslant}
\renewcommand{\geq}{\geqslant}

\swapnumbers
\theoremstyle{plain}
\newtheorem{theo}[subsection]{Theorem}%[section]
\newtheorem{lemm}[subsection]{Lemma}
\newtheorem{coro}[subsection]{Corollary}

\newtheorem{prop}[subsection]{Proposition}

\theoremstyle{definition}
\newtheorem{defi}[subsection]{Definition}

\newtheorem{rema}[subsection]{Remark}

\theoremstyle{remark}
\title{Dualities and intersection multiplicities}
\thanks{Preliminary version, \today}
\author{Anders J.\ Frankild}
\address{Anders J.\ Frankild\\ Department of Mathematical Sciences\\ University of Copenhagen\\ Universitetsparken 5\\
  2100 K{\o}benhavn \O\\ Denmark.}
\email{frankild@math.ku.dk}
\urladdr{http://www.math.ku.dk/~frankild/}
\author{Esben Bistrup Halvorsen}
\address{Esben Bistrup Halvorsen\\ Department of Mathematical Sciences\\ University of Copenhagen\\ Universitetsparken 5\\
  2100 K{\o}benhavn \O\\ Denmark.}
\email{esben@math.ku.dk}
\urladdr{http://www.math.ku.dk/~ebh/}
\thanks{The first author is a Steno stipendiat supported by FNU, the
  Danish Research Council}
\thanks{The second author is partially supported by FNU, the Danish
  Research Council }
\begin{document}
\begin{abstract}
Let $R$ be a commutative, noetherian, local ring. Topological $\QQ$--vector spaces
modelled on full subcategories of the derived category of $R$ are
constructed in order to study intersection multiplicities. 
\end{abstract}
\maketitle
\section{Introduction}
Let $R$ be a commutative, noetherian, local ring and let $X$ and $Y$
be homologically bounded complexes over $R$ with finitely generated
homology and supports intersecting at the maximal ideal. When the
projective dimension of $X$ or $Y$ is finite, their intersection
multiplicity is defined as
\[
\chi(X,Y)=\chi(X\ltensor_R Y),
\]
where $\chi(-)$ denotes the \emph{Euler characteristic} defined as the
alternating sum of the lengths of the homology modules. When $X$ and
$Y$ are modules, this definition agrees with the intersection
multiplicity defined by Serre~\cite{serre}.

The ring $R$ is said to  \emph{satisfy vanishing} when 
\begin{equation*}
\chi(X,Y) =0  \quad\text{provided $\dim(\supp X)+\dim(\supp Y)<\dim R$.}
\end{equation*}
If the above holds under the restriction that both complexes have
finite projective dimension, $R$ is said to \emph{satisfy weak
  vanishing}.

Assume, in addition, that $\dim(\supp
X)+\dim(\supp Y)\leq \dim R$ and that $R$ has prime characteristic
$p$. The \emph{Dutta multiplicity} of $X$ and $Y$ is defined when $X$
has finite projective dimension as the limit 
\[
\chi_\infty(X,Y) = \lim_{e\to\infty} \frac{1}{p^{e\codim(\supp
    X)}}\chi(\lfrob^e(X),Y),
\]
where $\lfrob^e$ denotes the $e$-fold composition of the left-derived
Frobenius functor; the Frobenius functor $\frob$ was systematically used in
the classical work by Peskine and Szpiro~\cite{peskineszpiro}. When $X$
and $Y$ are modules, $\chi_ \infty(X,Y)$ is the usual Dutta
multiplicity; see Dutta~\cite{duttamultiplicity}.

Let $\XXX$ be a specialization-closed subset of $\spec R$ and let
$\der(\XXX)$ denote the full subcategory of the derived category of $R$
comprising the homologically bounded complexes with finitely
generated homology and support contained in $\XXX$. The symbols
$\proj(\XXX)$ and $\inj(\XXX)$ denote the full subcategories of
$\der(\XXX)$ comprising the complexes that are isomorphic to a complex
of projective or injective modules, respectively.
The \emph{Grothendieck spaces}  $\grot\der(\XXX)$, $\grot\proj(\XXX)$ and 
$\grot\inj(\XXX)$ are topological $\QQ$--vector spaces modelled on these
categories. 
The first two of these spaces were introduced
in~\cite{halvorsenF} but were there 
modelled on ordinary non-derived categories of complexes. The
construction of Grothendieck spaces is similar to that of Grothendieck
groups but targeted at the study of intersection multiplicities.

The main
result of~\cite{halvorsenF} is a diagonalization theorem in prime
characteristic $p$ for the automorphism on $\grot\proj(\XXX)$ induced by the Frobenius
functor. A consequence of
this theorem is that every 
element $\alpha\in\grot\proj(\XXX)$ can be decomposed as
\[
\alpha=\alpha^{(0)}+\alpha^{(1)}+\cdots +\alpha^{(u)},
\]
where the component of degree zero describes the Dutta multiplicity,
whereas the components of 
higher degree describe the extent to which vanishing fails to hold for
the the
intersection multiplicity. This paper presents (see Theorem~\ref{injmaintheo}) a
similar diagonalization theorem for a functor that is analogous to the
Frobenius functor and has been studied by Herzog~\cite{herzog}. A
consequence is that every element $\beta\in\grot\inj(\XXX)$ can be
decomposed as
\[
\beta=\beta^{(0)}+\beta^{(1)}+\cdots +\beta^{(v)},
\]
where the component of degree zero describes an analog of the Dutta
multiplicity, whereas the components of higher degree describe the
extent to which vanishing fails to hold for the Euler form, introduced
by Mori and Smith~\cite{morismith}. Another consequence  (see 
Theorem~\ref{suffcondforweakvan}) is that $R$ satisfies weak vanishing
if only the Euler characteristic of homologically bounded complexes
with finite-length homology changes by a factor $p^{\dim R}$ when the analogous Frobenius functor
is applied. 

The duality functor $(-)^*=\rhom_R(-,R)$ on $\proj(\XXX)$ induces
an automorphism on $\grot\proj(\XXX)$ which in prime characteristic
$p$ is given by (see Theorem~\ref{dualitytheo})
\[
(-1)^{\codim\XXX}\alpha^*=\alpha^{(0)}-\alpha^{(1)}+\cdots +(-1)^u\alpha^{(u)}.
\]
Even in arbitrary characteristic, $R$ satisfies vanishing
if and only if all elements $\alpha\in\grot\proj(\XXX)$ are
\emph{self-dual} in the sense that 
$\alpha=(-1)^{\codim\XXX}\alpha^*$; and $R$ satisfies weak vanishing
if all elements $\alpha\in\grot\proj(\XXX)$ are \emph{numerically self-dual},
meaning that $\alpha-(-1)^{\codim\XXX}\alpha^*$ is in the
kernel of the homomorphism $\grot\proj(\XXX)\to\grot\der(\XXX)$
induced by the inclusion of the underlying categories (see
Theorem~\ref{selfdualityforvanishing}). Rings for which all elements of
the Grothendieck spaces $\grot\proj(\XXX)$ are
numerically self-dual include Gorenstein rings of dimension less than or equal to five
(see Proposition~\ref{Rgorenstein}) and
complete intersections (see Proposition~\ref{Rconditions} together with~\cite[Example~33]{halvorsenF}).

\section*{Notation}
\label{Notation}
Throughout,
$R$ denotes a commutative, noetherian, local ring with unique
maximal ideal $\mmm$ and residue field $k=R/\mmm$. Unless otherwise
stated, modules and complexes are assumed to be $R$--modules and
$R$--complexes, respectively.

\section{Derived categories and functors}
In this section we review notation and results from the theory of
derived categories, and we introduce a new star duality and derived
versions of the Frobenius functor and its natural analog.   For
details on the derived category and derived functors,
consult~\cite{gelfandmanin,hartshorne,verdier}. 

\subsection{Derived categories}
  A complex $X$ is a sequence $(X_i)_{i \in \ZZ}$ of
  modules equipped with a differential $(\partial^X_i)_{i\in\ZZ}$ lowering the homological
  degree by one.
  The homology complex $\H(X)$ of $X$ is the complex whose
  modules are 
\[
\H(X)_i=\H_i(X)=\ker \partial^X_i/\im\partial_{i+1}
\]
and whose differentials are trivial.

  A morphism of complexes $\sigma \colon X \to Y$ is a family
  $(\sigma_i)_{i \in \ZZ}$ of homomorphisms commuting with the
  differentials in $X$ and $Y$. The morphism of complexes $\sigma$ is a
  \emph{quasi-isomorphism} if the induced map on homology
  $\H_i(\sigma) \colon \H_i(X) \to \H_i(Y)$ is an isomorphism in every
  degree. Two morphisms of complexes $\sigma,\rho\colon X\to Y$ are
  \emph{homotopic} if there exists a family $(s_i)_{i\in\ZZ}$ of maps $s_i \colon X_i \to
  Y_{i+1}$ such that
  \begin{equation*}
     \sigma_i - \rho_i = \partial^Y_{i + 1}s_i + s_{i-1}\partial^X_i.
  \end{equation*}
  Homotopy yields an equivalence relation in the group $\hom_R(X,Y)$ of
  morphisms of complexes, and the \emph{homotopy category}
  $\mathsf{K}(R)$ is obtained from the category of complexes
  $\mathsf{C}(R)$ by declaring
  \begin{equation*}
    \hom_{\mathsf{K}(R)}(X,Y) = \hom_{\mathsf{C}(R)}(X,Y)/\,\text{homotopy}.
  \end{equation*}
  The collection $S$ of quasi-isomorphisms in the triangulated 
  category $\mathsf{K}(R)$ form a multiplicative system of morphisms.
  The \emph{derived category} $\derg(R)$ is obtained by (categorically)
  localizing $\mathsf{K}(R)$ with respect to $S$. Thus,
  quasi-isomorphisms become isomorphisms in $\derg(R)$; in the sequel,
  they are denoted $\simeq$.
  
  Let $n$ be an integer. The symbol $\shift^n{X}$ denotes the complex
  $X$ shifted (or translated or suspended) $n$ degrees to the left;
  that is, against the direction of the differential. The modules in
  $\shift^n{X}$ are given by $(\shift^{n}X)_i = X_{i-n}$, and the
  differentials are $\partial^{\shift^n X}_i =
  (-1)^{n}\partial^X_{i-n}$. The symbol $\sim$ denotes
  isomorphisms up to a shift in the derived category.

  The full subcategory of $\derg(R)$ consisting of complexes with
  bounded, finitely generated homology is denoted
  $\der(R)$. Complexes from $\der(R)$ are called \emph{finite}
  complexes. The symbols $\proj(R)$ and $\inj(R)$ denote the full
  subcategories of $\der(R)$ consisting of complexes that are
  isomorphic in the derived category to a bounded complex of projective modules and
  isomorphic to a bounded complex of injective modules,
  respectively.
  Note that $\proj(R)$ coincides with the full subcategory
  $\fla(R)$ of $\der(R)$ consisting of
  complexes isomorphic to a complex of flat modules.

\subsection{Support}
The \emph{spectrum} of $R$, denoted $\spec R$, is the set of prime
ideals of $R$. A subset $\XXX$ of $\spec R$ is
\emph{specialization-closed} if 
it has the property
\[
\ppp\in\XXX \text{ and }\ppp\subseteq\qqq\implies \qqq\in\XXX
\]
for all prime ideals $\ppp$ and $\qqq$.
A subset that is closed in the Zariski topology is, in particular, specialization-closed. 

The \emph{support} of a complex $X$ is the set
\[
    \supp X = \Big\{\,\ppp \in \spec R\,\Big\vert\, 
                      \H(X_{\ppp}) \neq 0 \,\Big\}.
\]
A \emph{finite} complex is a complex with bounded homology and finitely
generated homology modules; the support of such a complex is a closed and
hence specialization-closed subset
of $\spec R$. 

For a specialization-closed subset $\XXX$ of $\spec R$, the \emph{dimension} of $\XXX$,
denoted $\dim\XXX$, is the usual Krull dimension of $\XXX$. When $\dim
R$ is finite, the \emph{co-dimension} of $\XXX$, denoted $\codim\XXX$,
is the number $\dim R -\dim\XXX$. For a finitely generated module $M$,
the dimension and co-dimension of $M$, denoted $\dim M$ and $\codim
M$, are the dimension and co-dimension of the support of $M$. 

    For a specialization-closed subset
  $\XXX$ of $\spec R$, the symbols $\der(\XXX)$, $\proj(\XXX)$, and
  $\inj(\XXX)$ denote the full subcategories of $\der(R)$, $\proj(R)$,
  and $\inj(R)$, respectively, consisting of complexes whose support is contained
  in $\XXX$. In the case where $\XXX$ equals $\{\mmm\}$, we simply write
  $\der(\mmm)$, $\proj(\mmm)$ and $\inj(\mmm)$, respectively.

\subsection{Derived functors}
  A complex $P$ is said to be semi-projective if the functor $\hom_R(P,-)$
  sends surjective quasi-isomorphisms to surjective
  quasi-isomorphisms. If a complex is bounded to the right and
  consists of projective modules, it is semi-projective. A
  semi-projective resolution of $M$ is a quasi-isomorphism $\pi\colon
  P\to X$ where $P$ is semi-projective.
  
  Dually, a complex $I$ is said to be semi-injective if the functor
  $\hom_R(-,I)$ sends injective quasi-isomorphisms to surjective
  quasi-isomorphisms. If a  complex is bounded to the left and
  consists of injective modules, it is semi-injective. A semi-injective resolution of $Y$ is a
  quasi-isomorphism $\iota\colon Y\to I$ where $I$ is semi-injective.
  For existence of semi-projective and semi-injective resolutions see~\cite{DGA}.
  
  Let $X$ and $Y$ be complexes. The left-derived tensor product
  $X\ltensor_R Y$ in $\derg(R)$ of $X$ and $Y$ is defined by
  \begin{equation*}
     P \otimes_R Y \simeq X\ltensor_R Y \simeq X \otimes_R Q ,
  \end{equation*} 
  where $P \toeq X$ is a semi-projective resolution of $X$ and $Q \toeq
  Y$ is a semi-projective resolution of $Y$. The right-derived
  homomorphism complex $\rhom_R(X,Y)$ in $\derg(R)$ of $X$ and $Y$ is
  defined by 
  \begin{equation*}
    \hom_R(P,Y) \simeq \rhom_R(X,Y) \simeq \hom_R(X,I) ,
  \end{equation*} 
  where $P\toeq X$ is a semi-projective resolution of $X$ and $Y
  \toeq I$ is a semi-injective resolution of $Y$.  
  When $M$ and $N$ are modules,
  \begin{equation*}
  \H_n(M\ltensor_R N) \cong \tor^R_n(M,N)\quad\text{and}\quad 
  \H_{-n}(\rhom_R(M,N)) \cong \ext_R^n(M,N) 
  \end{equation*}
  for all integers $n$.

\subsection{Stability}
Let $\XXX$ and $\YYY$ be specialization-closed subsets of $\spec R$ and let $X$ be a
complex in $\der(\XXX)$ and $Y$ be a complex in $\der(\YYY)$. Then
\begin{align}\label{derfunc}
\begin{split}
X\ltensor_R Y \in\der(\XXX\cap\YYY) &\quad\text{if $X\in\proj(\XXX)$
  or $Y\in\proj(\YYY)$,}\\
X\ltensor_R Y \in\proj(\XXX\cap\YYY) &\quad\text{if $X\in\proj(\XXX)$
  and $Y\in\proj(\YYY)$,}\\
X\ltensor_R Y \in\inj(\XXX\cap\YYY) &\quad\text{if $X\in\proj(\XXX)$
  and $Y\in\inj(\YYY)$,}\\
X\ltensor_R Y \in\inj(\XXX\cap\YYY) &\quad\text{if $X\in\inj(\XXX)$
  and $Y\in\proj(\YYY)$,}\\ 
\rhom_R(X,Y) \in\der(\XXX\cap\YYY) &\quad\text{if $X\in\proj(\XXX)$
  or $Y\in\inj(\YYY)$,}\\
\rhom_R(X,Y) \in\proj(\XXX\cap\YYY) &\quad\text{if $X\in\proj(\XXX)$
  and $Y\in\proj(\YYY)$,}\\
\rhom_R(X,Y) \in\inj(\XXX\cap\YYY) &\quad\text{if $X\in\proj(\XXX)$
  and $Y\in\inj(\YYY)$\rlap{ and}}\\
\rhom_R(X,Y) \in\proj(\XXX\cap\YYY) &\quad\text{if $X\in\inj(\XXX)$
  and $Y\in\inj(\YYY)$.}
\end{split}
\end{align}

\subsection{Functorial isomorphisms}\label{iso}
  Throughout, we will make use of the functorial isomorphisms stated
  below. As we will not need them in the most general setting, the
  reader should bear in mind that not all the boundedness conditions
  imposed on the complexes are strictly necessary. For details the
  reader is referred e.g., to~\cite[A.4]{christensen} and the references
  therein.

  Let $S$ be another commutative, noetherian, local ring. Let
  $K,L,M\in\derg(R)$, let $P\in\derg(S)$ and let $N\in\derg(R,S)$, 
  the derived category of $R$--$S$--bi-modules.  There are the next functorial
  isomorphisms in $\derg(R,S)$.
  \begin{alignat*}{1}
    M\ltensor_R N & \toeq N\ltensor_R M. \tag{Comm}\\
    (M\ltensor_R N)\ltensor_S P&\toeq  M\ltensor_R(N \ltensor_S P). \tag{Assoc}\\
    \rhom_S(M\ltensor_R N,P) & \toeq \rhom_R(M,\rhom_S(N,P)). \tag{Adjoint}\\
    \rhom_R(M,\rhom_S(P,N)) & \toeq \rhom_S(P,\rhom_R(M,N)). \tag{Swap}
  \end{alignat*}
  Moreover, there are the following evaluation morphisms.
  \begin{alignat*}{1}
    \sigma_{KLP}\colon \rhom_R(K,L)\ltensor_S P
    &\to\rhom_R(K,L\ltensor_S P).
    \tag{Tensor-eval}\\
    \rho_{PLM}\colon P\ltensor_S \rhom_R(L,M) & \to\rhom_S(\rhom_R(P,L),M). 
    \tag{Hom-eval}
  \end{alignat*}
  In addition,
  \begin{itemize}
  \item the morphism $\sigma_{KLP}$ is invertible if $K$ is
    finite, $\H(L)$ is bounded, and either $P\in\projg(S)$ or
    $K\in\projg(R)$;\text{ and}
  \item the morphism $\rho_{PLM}$ is invertible if $P$ is
    finite, $\H(L)$ is bounded, and either $P\in\projg(R)$ or $M\in\injg(R)$.
  \end{itemize}

\subsection{Dualizing complexes}\label{dualizing}
  A finite complex $D$ is a \emph{dualizing complex} for
  $R$ if 
  \begin{equation*}
    D\in\inj(R) \quad\text{and}\quad R\toeq\rhom_R(D,D).
  \end{equation*}
  Dualizing complexes are essentially unique: if $D$ and $D'$ are
  dualizing complexes for $R$, then $D\sim D'$.
  To check whether a finite complex $D$ is dualizing is equivalent to
  checking whether
  \begin{equation*}
    k \sim \rhom_R(k,D).
  \end{equation*}
  A dualizing complex $D$ is said to be \emph{normalized} when $k \simeq
  \rhom_R(k,D)$.
  If $R$ is a Cohen--Macaulay ring of dimension $d$ and $D$ is a
  normalized dualizing complex, then $\H(D)$ is concentrated in degree
  $d$, and the module $\H_d(D)$ is the (so-called)
  \emph{canonical module}; see~\cite[Chapter~3]{brunsherzog}. 
 Observe that $\supp D=\spec R$.
  
  If $D$ is a normalized dualizing complex for $R$, then it is
  isomorphic to a complex
  \begin{equation*}
    0 \to D_{\dim R} \to D_{\dim R - 1} \to \cdots 
      \to D_1 \to D_0 \to 0
  \end{equation*}
  consisting of injective modules, where
  \begin{equation*}
    D_i =  \!\!\!\!\!\bigoplus_{\dim R/\ppp = i} \!\!\!\!\! E_R(R/\ppp)
  \end{equation*}
  and $E_R(R/\ppp)$ is the injective hull (or envelope) of $R/\ppp$ for a
  prime ideal $\ppp$; in particular, it follows that $D_0 = E_R(k)$.
  
  When $R$ is a homomorphic image of a local Gorenstein ring $Q$, then
  the $R$--complex $\shift^n\rhom_Q(R,Q)$, where $n=\dim Q-\dim R$, is a normalized dualizing
  complex over $R$.
  In particular, it follows from Cohen's structure theorem for
  complete local rings that any complete ring admits a dualizing
  complex.  Conversely, if a local ring admits a dualizing complex,
  then it must be a homomorphic image of a Gorenstein ring; this
  follows from Kawasaki's proof of Sharp's conjecture; see~\cite{kawasaki}.

\subsection{Dagger duality}\label{dagger}
  Assume that $R$ admits a normalized dualizing complex $D$ and
  consider the duality morphism of functors 
  \begin{equation*}
    \id_{\derg(R)} \to \rhom_R(\rhom_R(-,D),D).
  \end{equation*}
  It follows essentially from (Hom-eval) that the contravariant
  functor
  \begin{equation*}
  (-)^{\dagger} = \rhom_R(-,D)
  \end{equation*}
  provides a duality on the category $\der(R)$ which restricts to a
  duality between $\proj(R)$ and $\inj(R)$.  This duality is sometimes
  referred to as \emph{dagger duality}. According to~\eqref{derfunc}, if $\XXX$ is a
  specialization-closed subset of $\spec R$, then dagger duality gives a
  duality on $\der(\XXX)$ which restricts to a duality   between
  $\proj(\XXX)$ and $\inj(\XXX)$ as described by the following commutative
  diagram.
  \begin{displaymath}
    \xymatrix@C+50pt@R+20pt{
    {\der(\XXX)} \ar@<0.7ex>[r]^-{(-)^{\dagger}}
                & 
    {\der(\XXX)}\ar@<0.7ex>[l]^-{(-)^{\dagger}}\\
    {\proj(\XXX)}
    \ar@<0.7ex>[r]^-{(-)^{\dagger}} \ar[u]
                & 
    {\inj(\XXX).}\ar@<0.7ex>[l]^-{(-)^{\dagger}}\ar[u]\\
    }
  \end{displaymath}
  Here the vertical arrows are full embeddings of categories. For
  more details on dagger duality, see~\cite{hartshorne}.

\subsection{Foxby equivalence}\label{foxby}
  Assume that $R$ admits a normalized dualizing complex $D$ and consider the
 two contravariant adjoint functors
  \begin{equation*}
    D\ltensor_R - \quad\text{ and }\quad \rhom_R(D,-) ,
  \end{equation*}
  which come naturally equipped the unit and co-unit morphisms
  \begin{equation*}
    \eta \colon \id_{\derg(R)} \to \rhom_R(D,D\ltensor_R -) \quad\text{ and }\quad
    \epsilon \colon D\ltensor_R \rhom_R(D,-) \to \id_{\derg(R)}.
  \end{equation*}
  It follows essentially from an application of (Tensor-eval) and
  (Hom-eval) that the categories $\projg(R)$ and $\injg(R)$ are
  naturally equivalent via the above two functors. This
  equivalence is usually known as \emph{Foxby equivalence} and was
  introduced in~\cite{avramovfoxby}, to which the reader is referred for
  further details.
  
  According to~\eqref{derfunc}, for a specialization-closed subset $\XXX$ of $\spec R$, Foxby
  equivalence restricts to an equivalence between $\proj(\XXX)$ and
  $\inj(\XXX)$ as described by the following diagram.
  \begin{displaymath}
    \xymatrix@C+50pt@R+20pt{
      {\proj(\XXX)} \ar@<0.7ex>[r]^-{D\ltensor_R -}
      & {\inj(\XXX) .} \ar@<0.7ex>[l]^-{\rhom_R(D,-)}
    }
  \end{displaymath}

\subsection{Star duality}\label{star}
 Consider the duality morphism of functors
  \begin{equation*}
    \id_{\derg(R)} \to \rhom_R(\rhom_R(-,R),R).
  \end{equation*}
  From an application of (Hom-eval) it is readily seen that the functor
  \begin{equation*}
    (-)^* = \rhom_R(-,R)
  \end{equation*}
  provides a duality on the category $\proj(R)$. According to~\eqref{derfunc}, for a specialization-closed subset $\XXX$ of $\spec R$, 
star duality restricts to a duality on $\proj(\XXX)$ as described by following diagram.
  \begin{displaymath}
    \xymatrix@C+50pt@R+20pt{
      {\proj(\XXX)} \ar@<0.7ex>[r]^-{(-)^*}
      & {\proj(\XXX).} \ar@<0.7ex>[l]^-{(-)^*}
   }
  \end{displaymath}
  When $R$ admits a dualizing complex $D$, the star 
  functor can also be described in terms of the dagger and
  Foxby functors. Indeed, it is straightforward to show that the
  following three contravariant  endofunctors on $\proj(R)$ are isomorphic.
   \begin{equation*}%\label{starduality}
      (-)^* , \quad \rhom_R(D,-^\dagger) , \quad  \text{and}\quad (D\ltensor_R -)^{\dagger}.
   \end{equation*}
  It is equally straightforward to show that the following four contravariant
  endofunctors on $\inj(R)$ are isomorphic.
  \begin{equation*}
      (-)^{\dagger\,*\,\dagger}, \quad \rhom_R(D,-)^{\dagger} , \quad
      D\ltensor_R(\rhom_R(D,-)^*) \quad\text{and}\quad D\ltensor_R (-)^{\dagger}.
  \end{equation*}
 They provide a duality on $\inj(R)$. 
  In the sequel, the four isomorphic functors are denoted $(-)^{\star}$.
  According to~\eqref{derfunc}, for a specialization-closed subset $\XXX$ of $\spec R$, this new
  kind of star duality restricts to a duality on $\inj(\XXX)$ as
  described by the following diagram.
  \begin{displaymath}
    \xymatrix@C+50pt@R+20pt{
      {\inj(\XXX)} \ar@<0.7ex>[r]^-{(-)^\star}
      & {\inj(\XXX)}. \ar@<0.7ex>[l]^-{(-)^\star} 
      }
  \end{displaymath}
  
  The dagger duality, Foxby equivalence and star duality functors fit
  together in the following diagram.
\begin{equation}\label{dualitydiagram}
\begin{split}
\xymatrix@R+25pt@C=-12pt{
& {\der(\XXX)} \ar@<.6ex>[rr]^-{(-)^\dagger} & {\hspace{150pt}} & {\der(\XXX)}
\ar@<.6ex>[ll]^-{(-)^\dagger} \\
{\phantom{|}} \ar@(dl,ul)[]^{(-)^*}  & {\proj(\XXX)}  
\ar[u]  
\ar@<.6ex>[rr]^-{(-)^\dagger}
\ar@/^20pt/^-{D\ltensor_R -}[rr] & & {\,\inj(\XXX)} \ar[u]
\ar@<.6ex>[ll]^-{(-)^\dagger} \ar@/^20pt/[ll]^-{\rhom_R(D,-)} & {\phantom{|}} \ar@(dr,ur)[]_{(-)^\star}
}
\end{split}
\end{equation}
  In the lower part of the diagram, the three types of functors,
  dagger, Foxby and star, always commute pairwise, and the composition
  of two of the three types yields a functor of the third type. For
  example, star duality and dagger duality always commute and compose
  to give Foxby equivalence, since we have
  \[
  (-)^{*\dagger}\simeq (-)^{\dagger\star}\simeq D\ltensor_R -
  \quad\text{and}\quad (-)^{\star\dagger}\simeq (-)^{\dagger *} \simeq
  \rhom_R(D,-).
  \]

\subsection{Frobenius endofunctors}\label{endofunctors}
  Assume that $R$ is complete of prime characteristic $p$ and with perfect
  residue field $k$.  The
  endomorphism
  \begin{equation*}
    f \colon R \to R \quad\text{\, defined by \,}\quad f(r) = r^p
  \end{equation*} 
  for $r \in R$ is called the \emph{Frobenius endomorphism} on $R$.
  The $n$-fold composition of $f$, denoted $f^n$,
  operates on a generic element $r \in R$ by $f^n(r) = r^{p^n}$.
  We let $\f{n}R$ denote the $R$--algebra which, as a ring, is
  identical to $R$ but, as a module, is viewed through $f^n$. Thus,
  the $R$--module structure on $\f{n}R$ is given by
  \begin{equation*}
    r\cdot x =  r^{p^n}x \quad \text{for $r\in R$ and $x\in\f{n}R$} .
  \end{equation*}
  Under the present assumptions on $R$, the $R$--module $\f{n}R$
  is finitely generated (see, for example, Roberts~\cite[Section~7.3]{roberts}).
  
  We define two functors from the category of $R$--modules to the
  category of $\f{n}R$--modules by
  \begin{equation*}
    \frob^n(-)    = -\otimes_R \f{n}R \quad\text{\, and \,}\quad
    \herzog^n(-)  = \hom_R(\f{n}R,-),
  \end{equation*}
  where the resulting modules are finitely generated modules with
  $R$--structure obtained from the ring $\f{n}R=R$.  The functor
  $\frob^n$ is called the \emph{Frobenius functor} and 
  has been studied by Peskine and Szpiro~\cite{peskineszpiro}. The
  functor $\herzog^n$ has been studied by Herzog~\cite{herzog} and is
  analogous to $\frob^n$ in a sense that will be described below. We
  call this the \emph{analogous Frobenius functor}. The
  $R$--structure on $\frob^n(M)$ is given by
\[
r\cdot (m\otimes x) = m\otimes rx 
\]
for $r\in R$, $m\in M$ and $x\in\f{n}R$, and the $R$--structure on
$\herzog^n(N)$ is given by
\[
(r\cdot\phi)(x)= \phi(rx) 
\]
for $r\in R$, $\phi\in\hom_R(\f{n}R,N)$ and $x\in\f{n}R$. 
Note that here we also have
\[
(rm)\otimes x = m\otimes (r\cdot x)= m\otimes r^px\quad\text{and}\quad
r\phi(x) = \phi(r\cdot x) = \phi(r^px).
\]
Peskine and Szpiro~\cite[Th\'eor\`eme~(1.7)]{peskineszpiro} have
proven that, if $M$ has finite projective dimension, then so does
$\frob(M)$, and Herzog~\cite[Satz~5.2]{herzog} has proven that, if
$N$ has finite injective dimension, then so does $\herzog(N)$.

  It follows
  by definition that the functor $\frob^n$ is right-exact while the
  functor $\herzog^n$ is left-exact. We denote by $\lfrob^n(-)$ the
  left-derived of $\frob^n(-)$ and by $\rherzog^n(-)$ the
  right-derived of $\herzog^n(-)$. When $X$ and $Y$ are $R$--complexes
  with semi-projective and semi-injective resolutions 
  \begin{equation*}
   P \toeq X \quad\text{\, and \,}\quad Y \toeq I,
  \end{equation*}
  respectively, these derived functors are obtained as
  \begin{equation*}
    \lfrob^n(X)    = P\otimes_R \f{n}R \quad\text{\, and \,}\quad
    \rherzog^n(Y)  = \hom_R(\f{n}R,I),
  \end{equation*}
  where the resulting complexes are viewed through their
  $\f{n}R$--structure, which makes them $R$--complexes since $\f{n}R$
  as a ring is just $R$.  Observe that we may identify these functors
  with
  \begin{equation*}
     \lfrob^n(X)  = X \ltensor_R \f{n}R \quad\text{\, and \,}\quad
     \rherzog^n(Y)= \rhom_R(\f{n}R,Y). 
  \end{equation*}

\begin{lemm}\label{frobherzog}
  Let $R$ be a complete ring of prime characteristic and with perfect
  residue field, and let $\XXX$ be a specialization-closed subset of $\spec R$. Then
  the Frobenius functors commute with dagger and star duality in the
  sense that
  \begin{alignat*}{2}
    \lfrob^n(-)^{\dagger} &\simeq \rherzog^n(-^{\dagger}),
    & \rherzog^n(-)^{\dagger} &\simeq \lfrob^n(-^{\dagger}), \\
    \lfrob^n(-)^* &\simeq \lfrob^n(-^*) &\quad\text{and}\quad
    \rherzog^n(-)^\star &\simeq \rherzog^n(-^\star). 
  \end{alignat*}
  Here the first row contains isomorphisms of functors between $\proj(\XXX)$
  and $\inj(\XXX)$, while the second row contains isomorphisms of
  endofunctors on $\proj(\XXX)$ and $\inj(\XXX)$, respectively. Finally, the
  Frobenius functors commute with Foxby equivalence in the sense that
  \begin{align*}
    D\ltensor_R\lfrob^n(-) &\simeq \rherzog^n(D\ltensor_R -) \quad\text{and}\quad\\
    \rhom_R(D,\rherzog^n(-)) &\simeq \lfrob^n(\rhom_R(D,-)) 
  \end{align*}
  as functors from $\proj(\XXX)$ to $\inj(\XXX)$ and from
  $\inj(\XXX)$ to $\proj(\XXX)$, respectively.
\end{lemm}

\begin{proof}
  Let $\varphi \colon R \to S$ be a local homomorphism making $S$ into
  a finitely generated $R$--module, and let $D^R$ denote a normalized
  dualizing complex for $R$. Then $D^S = \rhom_R(S,D^R)$ is a
  normalized dualizing complex for $S$. 
  Pick an
  $R$--complex $X$ and consider the next string of natural
  isomorphisms.
  \begin{align*}
    \rhom_S(X \ltensor_R S, D^S) & \,\,\,= \,\,\rhom_S(X \ltensor_R S,\rhom_R(S,D^R))\\
    & \oteq \rhom_R(X \ltensor_R S,D^R)\\
    & \toeq \rhom_R(S,\rhom_R(X,D^R)).
  \end{align*}
  Here, the two isomorphism follow from (Adjoint). The computation
  shows that
  \begin{equation*}
    (- \ltensor_R S)^{\dagger_S} \simeq \rhom_R(S,-^{\dagger_R})
  \end{equation*}
  in $\derg(S)$. A similar computation using the natural isomorphisms
  (Adjoint) and (Hom-eval) shows that
  \[
   (-)^{\dagger_R}\ltensor_R S \simeq \rhom_R(S,-)^{\dagger_S}.
  \]
  Under the present assumptions, the $n$-fold composition of the
  Frobenius endomorphism $f^n \colon R \to R$ is module-finite map.
  Therefore, the above isomorphisms of functors yield
  \begin{equation*}
    \lfrob^n(-)^{\dagger} \simeq \rherzog^n(-^{\dagger})
    \quad\text{and}\quad  \lfrob^n(-^{\dagger}) \simeq \rherzog^n(-)^{\dagger}.
  \end{equation*}
  Similar considerations establish the remaining isomorphisms of functors.
\end{proof}

\begin{coro}\label{herzog}
  Let $R$ be a complete ring of prime characteristic and with perfect
  residue field, and let $\XXX$ be a specialization-closed subset of $\spec R$. Then
  the Frobenius functor $\rherzog^n$ is an endofunctor on
  $\inj(\XXX)$.
\end{coro}

\begin{proof}
  From the above lemma, we learn that 
  \begin{equation*}
    \rherzog^n(-) \simeq (-)^{\dagger}\circ\lfrob^n\circ(-)^{\dagger}
  \end{equation*}
  and since $\lfrob^n$ is an endofunctor on $\proj(\XXX)$ the
  conclusion is immediate.
\end{proof}

\begin{lemm}\label{square}
  Let $R$ be a complete ring of prime characteristic and with perfect
  residue field. For complexes $X,X' \in \proj(R)$ and $Y,Y' \in
  \inj(R)$ there are isomorphisms
  \begin{align*}
    \lfrob^n(X \ltensor_R X') &\simeq \lfrob^n(X)\ltensor_R\lfrob^n(X'),\\
    \rherzog^n(X \ltensor_R Y) &\simeq \lfrob^n(X)\ltensor_R\rherzog^n(Y),\\
\rherzog^n(\rhom_R(X,Y))   &\simeq \rhom_R(\lfrob^n(X),\rherzog^n(Y))\\
\lfrob^n(\rhom_R(X,X'))   &\simeq \rhom_R(\lfrob^n(X),\lfrob^n(X'))\rlap{\quad\text{and}}\\
\lfrob^n(\rhom_R(Y,Y'))   &\simeq \rhom_R(\rherzog^n(Y),\rherzog^n(Y'))  . 
  \end{align*}  
\end{lemm}

\begin{proof}
 We prove the first and the third isomorphism. The rest are obtained
 in a similar manner using Lemma~\ref{frobherzog} and the functorial isomorphisms.

 Let $F \toeq X$ and $F' \toeq X'$ be
  finite free resolutions. Then it follows
  \begin{align*}
    \lfrob^n(X \ltensor_R X') &\simeq
    \frob^n(F \otimes_R F') \\
    &\simeq
    \frob^n(F) \otimes_R \frob^n( F') \\
    &\simeq 
    \lfrob^n(X)\ltensor_R\lfrob^n(X').
  \end{align*}
  Here the first isomorphism follows as $F \otimes_R F'$ is isomorphic
  to $X \ltensor_R X'$; the second isomorphism follows from e.g.,
  \cite[Proposition~12(vi)]{halvorsenF}.

  From Corollary~\ref{herzog} we learn that
  \begin{equation*}
    \rherzog^n(Y) \simeq (\lfrob^n(Y^{\dagger}))^{\dagger},
  \end{equation*} 
  and therefore we may compute as follows.
  \begin{align*}
     \rhom_R(\lfrob^n(X),\rherzog^n(Y)) 
     & \simeq
     \rhom_R(\lfrob^n(X),(\lfrob^n(Y^{\dagger}))^{\dagger}) \\
     & \simeq 
     \rhom_R(\lfrob^n(X) \ltensor_R \lfrob^n(Y^{\dagger}),D) \\
     & \simeq
     \rhom_R(\lfrob^n(X \ltensor_R Y^{\dagger}),D) \\
     & \simeq 
     \lfrob^n(X \ltensor_R Y^{\dagger})^{\dagger} \\ 
     & \simeq 
     (\lfrob^n(\rhom_R(X,Y)^{\dagger})^{\dagger} \\
     & \simeq 
     \rherzog^n(\rhom_R(X,Y)). \\    
  \end{align*}
  Here the second isomorphism follows by (Adjoint); the third from the
  first statement in the Lemma; the fourth from definition; the fifth
  isomorphism follows from (Hom-eval); and the last isomorphism
  follows from Corollary~\ref{herzog}.
\end{proof}

\begin{rema}  
  Any complex in $\proj(R)$ is isomorphic to a bounded
  complex of finitely generated, free modules, and it is well-known
  that the Frobenius functor acts on such a complex by simply raising
  the entries in the matrices representing the differentials to the
  $p^n$'th power. To be precise, if $X$ is a complex in the form
  \[
X=   \cdots \longrightarrow R^m\overset{(a_{ij})}{\longrightarrow}
  R^n\longrightarrow\cdots\longrightarrow 0 ,
  \]
  then $\lfrob^n(X)=\frob^n(X)$ is a complex in the form
  \[
  \lfrob^n(X) =    \cdots \longrightarrow R^m\overset{(a^{p^n}_{ij})}{\longrightarrow}
   R^n\longrightarrow\cdots \longrightarrow 0.
  \]
  If $R$ is Cohen--Macaulay with canonical module $\omega$, then it
  follows from dagger duality that any
  complex in $\inj(R)$ is isomorphic to a complex $Y$ in the
  form
  \[
  Y=0\longrightarrow \cdots \longrightarrow \omega^n\overset{(a_{ji})}{\longrightarrow}
   \omega^m\longrightarrow\cdots ,
  \]
  and $\rherzog^n$ acts on $Y$ by raising
  the entries in the matrices
  representing the differentials to the $p^n$'th power, so that
  $\rherzog^n(Y)=\herzog^n(Y)$ is a complex in the form
  \[
  \rherzog^n(Y)=0\longrightarrow  \cdots \longrightarrow \omega^n\overset{(a^{p^n}_{ji})}{\longrightarrow}
   \omega^m\longrightarrow\cdots .
  \]
\end{rema}

\section{Intersection multiplicities}

\subsection{Serre's intersection multiplicity}\label{intmult}
  If $Z$ is a complex in $\der(\mmm)$, then its finitely many homology
  modules all have finite length, and the \emph{Euler characteristic}
  of $Z$ is defined by
  \begin{equation*}
    \chi(Z) = \sum_{i}(-1)^i\length \H_i(Z).
  \end{equation*}
  Let $X$ and $Y$ be finite complexes with
  $\supp X \cap \supp Y = \{\mmm\}$. The \emph{intersection multiplicity}
  of $X$ and $Y$ is defined by
  \begin{equation*}
    \chi(X,Y) = \chi(X\ltensor_R Y)\quad\text{when either $X\in\proj(R)$ or $Y\in\proj(R)$}.
  \end{equation*}
  In  the case where $X$ and $Y$ are finitely generated modules,
  $\chi(X,Y)$ coincides with Serre's intersection multiplicity; see~\cite{serre}.

  Serre's vanishing conjecture can be generalized to the statement that
  \begin{equation}\label{van}
    \chi(X,Y) = 0\quad\text{if $\dim(\supp X)  + \dim(\supp Y ) < \dim R$}
  \end{equation}
when either $X\in\proj(R)$ or $Y\in\proj(R)$. We will say that \emph{$R$ satisfies vanishing} when the above holds;
note that this, in general, is a stronger condition than Serre's
vanishing conjecture for modules.
It is known that $R$ satisfies vanishing in certain cases, for example when $R$
  is regular. However, it does not hold in general, as demonstrated by
  Dutta, Hochster and McLaughlin~\cite{dhm}.

If we require that both $X\in\proj(R)$ \emph{and}
  $Y\in\proj(R)$, condition \eqref{van} becomes weaker. When this
  weaker condition is satisfies, we say that 
  \emph{$R$ satisfies weak vanishing}. It is known that $R$ satisfies
  weak vanishing in many  cases, for example if
  $R$ is a complete intersection; see Roberts~\cite{robertsC} or Gillet and
  Soul{\'e}~\cite{gilletsoule}. There are, so far, no
  counterexamples preventing it from holding in full generality.
  
\subsection{Euler form}
  Let $X$ and $Y$ be finite complexes with $\supp X\cap\supp Y=\{\mmm\}$. The
  \emph{Euler form} of $X$ and $Y$ is defined by
  \begin{equation*}
    \xi(X,Y) = \chi(\rhom_R(X,Y))\quad\text{when either $X\in\proj(R)$ or $Y\in\inj(R)$}.
  \end{equation*}
  In  the case where $X$ and $Y$ are finitely generated modules,
  $\chi(X,Y)$ coincides with the Euler form introduced by Mori and
  Smith~\cite{morismith}.
  
  If $R$ admits a dualizing complex, then from Mori~\cite[Lemma~4.3(1)
  and~(2)]{moriE} and the definition of $(-)^\star$, we
  obtain
  \begin{align}\label{chixiformulas}
  \begin{split}
  \xi(X,Y) &= \chi(X,Y^\dagger)\quad  \text{whenever $X\in\proj(R)$
    or $Y\in\inj(R)$,}
  \\
  \chi(X^*,Y) &= \chi(X,Y^\dagger)\quad \text{whenever
    $X\in\proj(R)$,\rlap{\quad and}} \\
  \xi(X,Y^\star) &= \xi(X^\dagger,Y)\quad \text{whenever
    $Y\in\inj(R)$.}
  \end{split}
  \end{align}
Since the dagger functor does not change supports of complexes, the
first formula in \eqref{chixiformulas} shows that $R$ satisfies
vanishing exactly when
  \begin{equation}\label{xivan}
    \xi(X,Y) = 0\quad\text{if $\dim(\supp X)  + \dim(\supp Y ) < \dim R$}
  \end{equation}
when either $X\in\proj(R)$ or $Y\in\inj(R)$, and that $R$ satisfies
weak vanishing exactly when \eqref{xivan} holds when we  require both
$X\in\proj(R)$ \emph{and} $Y\in\inj(R)$.

\subsection{Dutta multiplicity}\label{duttamultiplicity}
Assume that $R$ is complete of prime characteristic $p$ and with
perfect residue field. Let $X$ and $Y$ be finite complexes with 
\[
\supp X\cap\supp Y=\{\mmm\}\quad\text{and}\quad \dim(\supp
X)+\dim(\supp Y)\leq\dim R.
\]
The \emph{Dutta multiplicity} of $X$ and $Y$ is defined by
  \begin{equation*}
    \chi_\infty(X,Y)= \lim_{e\to\infty}\frac{1}{p^{e\codim(\supp
        X)}}\chi(\lfrob^e(X),Y)\quad \text{when $X\in\proj(R)$.}
  \end{equation*}
When $X$ and $Y$ are finitely generated modules, $\chi_\infty(X,Y)$ coincides with the Dutta
  multiplicity defined in~\cite{duttamultiplicity}.

  The Euler form prompts to two natural analogs of the Dutta
  multiplicity.  We define
  \begin{gather*}
    \xi_\infty(X,Y) = \lim_{e\to\infty}
    \frac{1}{p^{e\codim(\supp Y)}}\xi(X,\rherzog^e(Y))\quad\text{when $Y\in\inj(R)$,\rlap{ and}}\\
   \xi^\infty(X,Y) = \lim_{e\to\infty}
   \frac{1}{p^{e\codim(\supp X)}}\xi(\lfrob^e(X),Y) \quad\text{when $X\in\proj(R)$}.
  \end{gather*}
  We immediately note, using~\eqref{chixiformulas} together
  with Lemma~\ref{frobherzog}, that
  \begin{gather*}
    \xi_\infty(X,Y) = \chi_\infty(Y^\dagger,X)\quad \text{whenever
      $Y\in\inj(\YYY)$,\rlap{ and}}
    \\
    \xi^\infty(X,Y) = \chi_\infty(X^*,Y)\quad \text{whenever
      $X\in\proj(\XXX)$.}
  \end{gather*}

\section{Grothendieck spaces} 
In this section we present the definition and basic properties of Grothendieck
spaces.  We will introduce three types of Grothendieck spaces, two of
which were introduced in~\cite{halvorsenF}.
The constructions in \emph{loc.\ cit.}\ are  different
from the ones here but yield the same spaces. 

\subsection{Complement}
For any specialization-closed subset $\XXX$ of $\spec R$, a new subset
is defined by 
  \begin{equation*}
    \comp{\XXX} = \Big\{\, \ppp \in \spec{R}\,\Big\vert\,\XXX \cap \var(\ppp) = \{\mmm\}
              \text{ and } \dim\var(\ppp) \leq \codim \XXX \,\Big\} .
  \end{equation*}  
  This set is engineered to be the largest subset of $\spec R$ such that
  \begin{equation*}
    \XXX \cap \comp{\XXX} = \{\mmm\} \quad\text{\, and \,}\quad
    \dim \XXX + \dim\comp{\XXX} \leq \dim R. 
  \end{equation*}
In fact, when $\XXX$ is closed,
\[
\dim\XXX +\dim\comp\XXX=\dim R.
\]
Note that $\comp\XXX$ is specialization-closed and that $\XXX\subseteq\compcomp\XXX$.

 \subsection{Grothendieck space}
  Let $\XXX$ be a specialization-closed subset of $\spec{R}$. The \emph{Grothendieck
    space} of the category $\proj(\XXX)$ is the $\QQ$--vector space
  $\grot\proj(\XXX)$ presented by elements $[X]_{\proj(\XXX)}$, one for
  each isomorphism class of a complex $X\in\proj(\XXX)$, and relations
  \begin{equation*}
    [X]_{\proj(\XXX)} = [\tilde X]_{\proj(\XXX)} \quad\text{whenever}\quad
    \chi(X, -)=\chi(\tilde X, -)
  \end{equation*}
  as metafunctions (``functions'' from a category to a set)
  $\der(\comp{\XXX})\to \QQ$.

Similarly, the Grothendieck
  space of the category $\inj(\XXX)$ is the $\QQ$--vector space
  $\grot\inj(\XXX)$ presented by elements $[Y]_{\inj(\XXX)}$, one for each
  isomorphism class of a complex $Y\in\inj(\XXX)$, and relations
  \begin{equation*}
   [Y]_{\inj(\XXX)} = [\tilde Y]_{\inj(\XXX)} \quad\text{whenever}\quad
   \xi(-,Y)=\xi(-, \tilde Y)
  \end{equation*}
as metafunctions $\der(\comp{\XXX})\to \QQ$.

Finally, the Grothendieck
  space of the category $\der(\XXX)$ is the $\QQ$--vector space
  $\grot\der(\XXX)$ presented by elements $[Z]_{\der(\XXX)}$, one for each
  isomorphism class of a complex $Z\in\der(\XXX)$, and relations
  \begin{equation*}
   [Z]_{\der(\XXX)} = [\tilde Z]_{\der(\XXX)} \quad\text{whenever}\quad
   \chi(-,Z)=\chi(-, \tilde Z)
  \end{equation*}
as metafunctions $\proj(\comp{\XXX})\to \QQ$. Because of~\eqref{chixiformulas},
  these relations are exactly the same as the relations
  \begin{equation*}
   [Z]_{\der(\XXX)} = [\tilde Z]_{\der(\XXX)} \quad\text{whenever}\quad
   \xi(Z,-)=\xi(\tilde Z,-)
  \end{equation*}
as metafunctions $\inj(\comp{\XXX})\to \QQ$.
  
  By definition of the Grothendieck space $\grot\proj(\XXX)$ there is, for each complex $Z$
  in $\der(\comp\XXX)$, a well-defined $\QQ$--linear map
  \begin{equation*}
   \chi(-,Z)\colon \grot\proj(\XXX)\to\QQ\quad\text{given by}\quad
   [X]_{\proj(\XXX)}\mapsto\chi(X,Z) .
  \end{equation*}
  We equip $\grot\proj(\XXX)$ with the \emph{initial topology} induced
  by the  family of maps in the above form. This topology is the coarsest
  topology on $\grot\proj(\XXX)$ making the above map continuous for
  all $Z$ in $\der(\comp\XXX)$. Likewise, for each complex $Z$ in
  $\der(\comp\XXX)$, there is a well-defined $\QQ$--linear map
  \begin{equation*}
   \xi(Z,-)\colon \grot\inj(\XXX)\to\QQ\quad\text{given by}\quad
   [Y]_{\inj(\XXX)}\mapsto\xi(Z,Y) ,
  \end{equation*}
  and we equip $\grot\inj(\XXX)$ with the initial topology induced by the
  family of maps in the above form. Finally, for each complex $X$ in
  $\proj(\comp\XXX)$, there is a well-defined $\QQ$--linear map
  \begin{equation*}
   \chi(X,-)\colon \grot\der(\XXX)\to\QQ\quad\text{given by}\quad
   [Z]_{\der(\XXX)}\mapsto\chi(X,Z) ,
  \end{equation*}
  and we equip $\grot\der(\XXX)$ with the initial topology induced by the
  family of maps in the above form.
By~\eqref{chixiformulas}, this topology is the same as the initial
topology induced by the family of (well-defined, $\QQ$--linear) maps
in the form
  \begin{equation*}
   \xi(-,Y)\colon \grot\der(\XXX)\to\QQ\quad\text{given by}\quad
   [Z]_{\der(\XXX)}\mapsto\xi(Z,Y) ,
  \end{equation*}
for complexes $Y$ in $\inj(\comp\XXX)$.

  It is straightforward to see that addition
  and scalar multiplication are continuous operations on Grothendieck
  spaces, making $\grot\proj(\XXX)$, $\grot\der(\XXX)$ and $\grot\inj(\XXX)$
  topological $\QQ$--vector spaces. We shall always consider
  Grothendieck spaces as topological $\QQ$--vector spaces, so that, for example, a
  ``homomorphism'' between Grothendieck spaces means a homomorphism of
  topological $\QQ$--vector spaces: that is, a continuous,
  $\QQ$--linear map.

The following proposition is an improved version of~\cite[Proposition~2(iv) and~(v)]{halvorsenF}.

\begin{prop}\label{grothendieckobs}
  Let $\XXX$ be a specialization-closed subset of $\spec R$.
  \begin{enumerate}
  \item \label{intheform} Any element in $\grot\proj(\XXX)$ can be
    written in the form $r[X]_{\proj(\XXX)}$ for some $r\in\QQ$ and
    some $X\in\proj(\XXX)$, any element in $\grot\inj(\XXX)$ can be
    written in the form $s[Y]_{\inj(\XXX)}$ for some $s\in\QQ$ and
    some $Y\in\inj(\XXX)$, and any element in $\grot\der(\XXX)$ can be
    written in the form $t[Z]_{\der(\XXX)}$ for some $t\in\QQ$ and
    some $Z\in\der(\XXX)$. Moreover, $X$, $Y$ and $Z$ may be chosen so that
    \begin{equation*}
      \codim(\supp X)=\codim(\supp Y)=\codim(\supp Z)=\codim\XXX .
    \end{equation*}
  \item \label{generatedg} For any complex $Z\in\der(\XXX)$, we have the
    identity
    \begin{equation*}
      [Z]_{\der(\XXX)} = [\H(Z)]_{\der(\XXX)}.
    \end{equation*}
In particular, the $\QQ$--vector space $\grot\der(\XXX)$ is generated
by elements in the form $[R/\ppp]_{\der(\XXX)}$ for prime ideals $\ppp$ in $\XXX$. 
\end{enumerate}
\end{prop}
\begin{proof}
\eqref{intheform} By construction, any element $\alpha$ in
$\grot\proj(\XXX)$ is a $\QQ$--linear combination
\[
\alpha=r_1[X^1]_{\proj(\XXX)}+\cdots +r_n[X^n]_{\proj(\XXX)}
\]
where $r_i\in\QQ$ and $X^i\in\proj(\XXX)$. Since a shift of a complex changes the sign
of the corresponding element in the Grothendieck space, 
we can assume that $r_i>0$ for all $i$. Choosing a greatest common
denominator for the $r_i$'s, we can find  $r\in\QQ$ such that 
\[
\alpha = r(m_1[X^1]_{\proj(\XXX)}+\cdots +m_n[X^n]_{\proj(\XXX)}) = r[X]_{\proj(\XXX)},
\]
where the $m_i$'s are natural numbers and $X$ is the direct sum over
$i$ of $m_i$ copies of $X^i$.

In order to prove the last statement of~\eqref{intheform}, choose a
prime ideal $\ppp=(a_1,\dots ,a_t)$ in $\XXX$ which is first in a
chain
$\ppp=\ppp_0\subsetneq\ppp_1\subsetneq\cdots\subsetneq\ppp_t=\mmm$ of
prime ideals in $\XXX$ of maximal length $t=\codim\XXX$. Note that $\XXX\supseteq\var(\ppp)$ and
that the Koszul complex $K=K(a_1,\dots ,a_t)$ has support exactly
equal to $\var(\ppp)$. It follows that
\[
\alpha=\alpha+0=r[X]_{\proj(\XXX)}+r[K]_{\proj(\XXX)}-r[K]_{\proj(\XXX)}=r[X\oplus
K\oplus\shift K]_{\proj(\XXX)},
\]
where $\codim(\supp(X\oplus K\oplus\shift K))=\codim\XXX$. The
same argument applies to elements of $\grot\inj(\XXX)$ and $\grot\der(\XXX)$.

\eqref{generatedg} 
Any complex in
$\der(\XXX)$ is
isomorphic to a bounded complex. After an appropriate shift, we may
assume that $Z$ is 
a complex in $\der(\XXX)$ in the form
\[
0\to Z_n\to \cdots \to Z_1\to Z_0 \to 0 
\]
for some natural number $n$.
Since $\H_n(Z)$ is the kernel of the map $Z_n\to Z_{n-1}$, we can construct a
short exact sequence of complexes
\[
0\to \shift^n \H_n(Z) \to Z \to Z' \to 0,
\]
where $Z'$ is a complex in $\der(\XXX)$ concentrated in the same
degrees as $Z$. The complex $Z'$ is exact in degree $n$, and $\H_i(Z')=\H_i(Z)$ for
$i=n-1,\dots ,0$. In the Grothendieck space $\grot\der(\XXX)$, we then
have
\[
[Z]_{\der(\XXX)} = [\shift^n\H_n(Z)]_{\der(\XXX)} + [Z']_{\der(\XXX)} .
\]
Again, 
$Z'$ is isomorphic to a
complex concentrated in degree $n-1,\cdots ,0$, so we can repeat the
process a finite number of times and achieve that 
\begin{align*}
[Z]_{\der(\XXX)} &= [\shift^n\H_n(Z)]_{\der(\XXX)} + \cdots +
[\shift\H_1(Z)]_{\der(\XXX)} + [\H_0(Z)]_{\der(\XXX)} \\
&= [\shift^n\H_n(Z)\oplus \cdots
\oplus\shift\H_1(Z)\oplus\H_0(Z)]_{\der(\XXX)} \\
&= [\H(Z)]_{\der(\XXX)} .
\end{align*}
The above analysis shows that any element of $\grot\der(\XXX)$ can be written
in the form
\[
r[Z]_{\der(\XXX)}=r\sum_i(-1)^i [\H_i(Z)]_{\der(\XXX)},
\]
which means that $\grot\der(\XXX)$ is generated by modules. Taking a
filtration of a module 
establishes that $\grot\der(\XXX)$ must be generated by elements 
of the form $[R/\ppp]_{\der(\XXX)}$ for prime ideals $\ppp$ in $\XXX$. 
\end{proof}

\subsection{Induced Euler characteristic}
  The Euler characteristic $\chi\colon\der(\mmm)\to\QQ$ induces an
  isomorphism\footnote{That is, a $\QQ$--linear homeomorphism.}
  \begin{equation}\label{chiisomorphism}
    \grot\der(\mmm)\overset{\cong}{\longrightarrow}\QQ\quad\text{given by}\quad
    [Z]_{\der(\mmm)}\mapsto\chi(Z).
  \end{equation}
See~\cite{halvorsenF} for more details.
  We also denote this isomorphism by $\chi$. The isomorphism means
  that we can identify the intersection multiplicity $\chi(X,Y)$ and
  the Euler form $\xi(X,Y)$ of complexes $X$ and $Y$ with elements in
  $\grot\der(\mmm)$ of the form
  \begin{equation*}
    [X\ltensor_R Y]_{\der(\mmm)} \quad\text{and}\quad
    [\rhom_R(X,Y)]_{\der(\mmm)} ,
  \end{equation*}
respectively.
\subsection{Induced inclusion}
  Let $\XXX$ be a specialization-closed subset of $\spec R$. It is
  straightforward to verify that the full embeddings of $\proj(\XXX)$ and $\inj(\XXX)$ into $\der(\XXX)$ induce
  homomorphisms\footnote{That is, continuous, $\QQ$--linear maps.}
  \begin{align*}
   \grot\proj(\XXX)\to\grot\der(\XXX) \quad &\text{given by}\quad
   [X]_{\proj(\XXX)}\mapsto [X]_{\der(\XXX)},\text{\rlap{ and}}\\
   \grot\inj(\XXX)\to\grot\der(\XXX) \quad &\text{given by}\quad
   [Y]_{\inj(\XXX)}\mapsto [Y]_{\der(\XXX)}.
  \end{align*}
  If $\XXX$ and $\XXX'$ are specialization-closed subsets of $\spec R$ such that
  that $\XXX\subseteq\XXX'$, then it is straightforward to verify that
  the full embeddings of $\proj(\XXX)$ into
  $\proj(\XXX')$, $\inj(\XXX)$ into $\inj(\XXX')$ and $\der(\XXX)$
  into $\der(\XXX')$  induce  homomorphisms
  \begin{align*}
    \grot\proj(\XXX)\to\grot\proj(\XXX')
    \quad &\text{given by}\quad [X]_{\proj(\XXX)} \mapsto
    [X]_{\proj(\XXX')}, \\
    \grot\inj(\XXX)\to\grot\inj(\XXX')
    \quad &\text{given by}\quad [Y]_{\inj(\XXX)} \mapsto
    [Y]_{\inj(\XXX')}, \text{\rlap{ and}} \\
    \grot\der(\XXX)\to\grot\der(\XXX')
    \quad &\text{given by}\quad [Z]_{\der(\XXX)} \mapsto
    [Z]_{\der(\XXX')}.
  \end{align*}
    The maps obtained
  in this way are called \emph{inclusion homomorphisms}, and we shall
  often denote them by an overline: if $\sigma$ is an element in a Grothendieck space, then
  $\overline\sigma$ denotes the image of $\sigma$ after an application
  of an inclusion homomorphisms.

\subsection{Induced tensor product and Hom}
Proposition~\ref{inducedhomomorphisms} below shows that the
left-derived tensor product functor and the right-derived Hom-functor
induce bi-homomorphisms\footnote{That is, maps that are continuous and
 $\QQ$--linear in each variable.} on Grothendieck spaces.
To clarify the contents of the proposition, let $\XXX$ and $\YYY$ be
specialization-closed subsets of $\spec R$ such that 
$\XXX\cap\YYY=\{\mmm\}$ and $\dim\XXX+\dim\YYY\leq\dim R$.
Proposition~\ref{inducedhomomorphisms} states, for example, that the right-derived
Hom-functor induces a bi-homomorphism
\[
\hom\colon\grot\proj(\XXX)\times\grot\inj(\YYY) \to \grot\inj(\mmm).
\]
Given elements $\sigma\in\grot\proj(\XXX)$ and $\tau\in\grot\inj(\YYY)$, we can, by Proposition~\ref{grothendieckobs}, write
\begin{equation*}
  \sigma=r[X]_{\proj(\XXX)}\quad\text{and}\quad \tau=s[Y]_{\inj(\YYY)},
\end{equation*} 
where $r$ and $s$ are rational numbers, $X$ is a complex in
$\proj(\XXX)$ and $Y$ is a complex in $\inj(\YYY)$. The
bi-homomorphism above is then given by
\begin{equation}\label{indhomex}
(\sigma,\tau)\mapsto \hom(\sigma,\tau) =
rs[\rhom_R(X,Y)]_{\der(\mmm)} .
\end{equation}
We shall use the symbol ``$\otimes$'' to denote any bi-homomorphism on
Grothendieck spaces induced by the left-derived tensor product and
the symbol ``$\hom$'' to denote any bi-homomorphism induced by
right-derived Hom-functor. 
Together with the isomorphism in~\eqref{chiisomorphism} it follows
that the intersection multiplicity $\chi(X,Y)$ and Euler form
$\xi(X,Y)$ can be identified with elements in $\grot\der(\mmm)$ of the
form
\begin{gather*}
[X]_{\proj(\XXX)} \otimes  [Y]_{\der(\YYY)}, \qquad\qquad
[X]_{\der(\XXX)} \otimes [Y]_{\proj(\YYY)},\\
\hom([X]_{\der(\XXX)} ,[Y]_{\inj(\YYY)}) \quad\text{and}\quad 
\hom([X]_{\proj(\XXX)} ,[Y]_{\der(\YYY)}) .
\end{gather*}

\begin{prop}\label{inducedhomomorphisms}
  Let $\XXX$ and $\YYY$ be specialization-closed subsets of $\spec R$ such that
  $\XXX\cap\YYY=\{\mmm\}$ and $\dim\XXX+\dim\YYY\leq\dim R$. The
  left-derived tensor product induces bi-homomorphisms as in the first
  column below, and the right-derived Hom-functor induces
  bi-homomorphisms as in the second column below.
\begin{alignat*}{2}
\grot\proj(\XXX)\times \grot\der(\YYY) &\to \grot\der(\mmm), & \grot\proj(\XXX)\times \grot\der(\YYY) &\to \grot\der(\mmm),\\
\grot\der(\XXX)\times \grot\proj(\YYY) &\to \grot\der(\mmm), & \grot\der(\XXX)\times \grot\inj(\YYY) &\to \grot\der(\mmm),\\
\grot\proj(\XXX)\times \grot\proj(\YYY)&\to \grot\proj(\mmm), & \grot\proj(\XXX)\times \grot\inj(\YYY) &\to \grot\inj(\mmm), \\
\grot\proj(\XXX)\times \grot\inj(\YYY) &\to \grot\inj(\mmm), & \grot\proj(\XXX)\times \grot\proj(\YYY) &\to \grot\proj(\mmm),\\
\grot\inj(\XXX)\times \grot\proj(\YYY) &\to \grot\inj(\mmm)\quad\text{and}\quad & \grot\inj(\XXX)\times \grot\inj(\YYY) &\to \grot\proj(\mmm).
\end{alignat*}
\end{prop}

\begin{proof}
  We verify that the map
  \[
    \hom\colon \grot\proj(\XXX)\times \grot\inj(\YYY)\to
    \grot\inj(\mmm)
  \]
  given as in~\eqref{indhomex} is a well-defined bi-homomorphism, leaving the
  same verifications for the remaining maps as an easy exercise for
  the reader.
  
  Therefore, assume that $X$ and $\tilde X$ are complexes from
  $\proj(\XXX)$ and that $Y$ and $\tilde Y$ are complexes from
  $\inj(\YYY)$ such that
  \begin{equation*}
   \sigma = [X]_{\proj(\XXX)}=[\tilde X]_{\proj(\XXX)}
   \quad\text{and}\quad \tau = [Y]_{\inj(\YYY)}=[\tilde Y]_{\inj(\YYY)} .
  \end{equation*}
In order to show that the map is a well-defined
  $\QQ$--bi-linear map, we are required to demonstrate that
  \begin{equation*}
    [\rhom_R(X,Y)]_{\inj(\mmm)} = [\rhom_R(\tilde X,\tilde
    Y)]_{\inj(\mmm)}.
  \end{equation*}
  To this end, let $Z$ be an arbitrary complex in
  $\der(\comp{\{\mmm\}})=\der(R)$. We want to show that
\[
\xi(Z,\rhom_R(X,Y))=\xi(Z,\rhom_R(\tilde X,\tilde Y)).
\]
Without loss of generality, we may assume that $R$
is complete; in particular, we may assume that $R$ admits a normalized dualizing
complex. Observe that 
  \begin{equation*}
    Z\otimes_R X\in\der(\XXX)\subseteq\der(\comp\YYY)\quad\text{and}\quad 
    Z\ltensor_R Y^\dagger\in\der(\YYY)\subseteq \der(\comp{\XXX}).
  \end{equation*}
  Applying \eqref{chixiformulas}, (Hom-eval) and (Assoc), we learn
  that
  \begin{align}\label{xichi}
\begin{split}
   \xi(Z,\rhom_R(X,Y)) & = \chi(Z,\rhom_R(X,Y)^\dagger) \\
    &=  \chi(Z,X\ltensor_R Y^\dagger) \\
    &= \chi(X,Z \ltensor_R Y^\dagger).
\end{split}
  \end{align}
  A similar computation shows that $\xi(Z,\rhom_R(\tilde
  X,Y))=\chi(\tilde X,Z \ltensor_R Y^\dagger)$, and since
  $[X]_{\proj(\XXX)}=[\tilde X]_{\proj(\XXX)}$, we conclude
  that
  \[
    \xi(Z,\rhom_R(X,Y)) = \xi(Z,\rhom_R(\tilde X,Y)).
  \]
  An application of (Adjoint) yields that
  \[
   \xi(Z,\rhom_R(\tilde X,Y))  = \xi(Z\ltensor_R \tilde X,Y),
  \]
  and similarly $\xi(Z,\rhom_R(\tilde X,\tilde Y)) = \xi(Z\ltensor_R
  \tilde X,\tilde Y)$. Since $[Y]_{\inj(\YYY)}=[\tilde Y]_{\inj(\YYY)}$,
  we conclude that
  \[
   \xi(Z,\rhom_R(\tilde X,Y)) = \xi(Z,\rhom_R(\tilde X,\tilde Y)).
  \]
  Thus, we have that
  \[
    \xi(Z,\rhom_R(X,Y))= \xi(Z,\rhom_R(\tilde X,\tilde Y)),
  \]
  which establishes well-definedness.
  
  By definition, the induced Hom-map is $\QQ$--linear. To establish
  that it is continuous in, say, the first variable it suffices for
  fixed $\tau\in\grot\inj(\YYY)$ to show
  that, to every $\epsilon > 0$ and every complex $Z\in\der(\comp{\{\mmm\}})=\der(R)$, there
  exists a $\delta > 0$ and a complex $Z'\in\der(\comp\XXX)$ such that
  \begin{equation*}
    |\chi(\sigma,Z')|<\delta \implies
    |\xi(Z,\hom(\sigma,\tau))|<\epsilon .
  \end{equation*}
We can write $\tau=r[Y]_{\inj(\YYY)}$ for an
$Y\in\inj(\YYY)$ and a rational number $r>0$. According
to~\eqref{xichi}, the implication above is then achieved with $Z' =
  Z\ltensor_R Y^\dagger$ and $\delta=\epsilon/r$. Continuity in the
  second variable is shown by similar arguments.
\end{proof}

In Proposition~\ref{inducedisos} below, we will show that the dagger,
  Foxby and star functors from diagram~\eqref{dualitydiagram} induce
  isomorphisms of Grothendieck spaces. We shall denote 
  the isomorphisms induced by
  the star and dagger duality functors by the same symbol as the
  original functor, whereas the isomorphisms induced by the Foxby
  functors will be denoted according to
  Proposition~\ref{inducedhomomorphisms} by $D\otimes -$ and
  $\hom(D,-)$. In this way, for example,
\[
[X]^\dagger_{\proj(\XXX)}=[X^\dagger]_{\inj(\XXX)},\quad
[X]^*_{\proj(\XXX)}=[X^*]_{\proj(\XXX)}\quad\text{and}\quad D\otimes
    [X]_{\proj(\XXX)} = [D\ltensor_R X]_{\inj(\XXX)}.
\]

\begin{prop}\label{inducedisos}
  Let $\XXX$ be a specialization-closed subset of $\spec R$, and assume that $R$
  admits a dualizing complex. The functors from diagram~\eqref{dualitydiagram}
induce isomorphisms of Grothendieck spaces as described by the
horizontal and circular arrows in the following commutative diagram. 
\begin{equation*}
\xymatrix@R+25pt@C=-12pt{
& {\grot\der(\XXX)} \ar@<.6ex>[rr]^-{(-)^\dagger} & {\hspace{150pt}} & {\grot\der(\XXX)}
\ar@<.6ex>[ll]^-{(-)^\dagger} \\
{\phantom{|}} \ar@(dl,ul)[]^{(-)^*}  & {\grot\proj(\XXX)}  
\ar[u]  
\ar@<.6ex>[rr]^-{(-)^\dagger}
\ar@/^20pt/^-{D\ltensor_R -}[rr] & & {\,\grot\inj(\XXX)} \ar[u]
\ar@<.6ex>[ll]^-{(-)^\dagger} \ar@/^20pt/[ll]^-{\rhom_R(D,-)} & {\phantom{|}} \ar@(dr,ur)[]_{(-)^\star}
}
\end{equation*}
\end{prop}

\begin{proof}
The fact that the dagger, star and Foxby functors induce homomorphisms
on Grothendieck spaces follows immediately from
Proposition~\ref{inducedhomomorphisms}. The fact that the induced
homomorphisms are isomorphisms follows immediately from~\ref{dagger},
\ref{foxby} and~\ref{star}, since the underlying functors define
dualities or equivalences of categories.
\end{proof}

\begin{prop} \label{formulas}
  Let $\XXX$ be a specialization-closed subset of $\spec R$ and consider the
  following elements of Grothendieck spaces.
  \begin{equation*}
    \alpha\in\grot\proj(\XXX), \quad
    \beta\in\grot\inj(\XXX), \quad
    \gamma\in\grot\der(\comp{\XXX}) \quad\text{and}\quad
    \sigma\in\grot\der(\mmm).
  \end{equation*} 
  Then $\sigma^\dagger=\sigma$ holds in $\grot\der(\mmm)$, and so do
  the following identities.
  \begin{align*}
    \alpha\otimes\gamma=\hom(\gamma,\alpha^\dagger)&=\hom(\alpha,\gamma^\dagger)=\hom(\alpha^*,\gamma)\\
    \hom(\alpha,\gamma)=\alpha\otimes\gamma^\dagger&=\hom(\gamma,D\otimes\alpha)=\alpha^*\otimes\gamma\\
    \hom(\gamma,\beta)=\beta^\dagger\otimes\gamma&=\hom(\hom(D,\beta),\gamma) \\
    \hom(\beta^\dagger,\gamma)=\hom(\gamma^\dagger,\beta)&=\hom(D,\beta)\otimes\gamma=\hom(\gamma,\beta^\star)
  \end{align*}
\end{prop}

\begin{proof}
Recall from~\ref{star} that the Foxby functors can be written as the composition of
a star and a dagger functor. All identities follow from the formulas
in~\eqref{chixiformulas}. The formula for $\sigma$ is a consequence
of the first formula in~\eqref{chixiformulas} in the case $X=R$.
\end{proof}

\subsection{Frobenius endomorphism}
  Assume that $R$ is complete of prime characteristic $p$ and with
  perfect residue field. Let $\XXX$ be a specialization-closed subset of $\spec{R}$,
  and let $n$ be a non-negative integer. The derived Frobenius endofunctor
  $\lfrob^n$ on $\proj(\XXX)$
  induces an endomorphism\footnote{That is, a continuous, $\QQ$--linear
  operator.} on $\grot\proj(\XXX)$,
  which will be denoted $\frob^n_\XXX$; see~\cite{halvorsenF} for
  further details. It is given for a complex $X\in\proj(\XXX)$ by
\[
\frob^n_\XXX([X]_{\proj(\XXX)})=[\lfrob^n(X)]_{\proj(\XXX)}.
\]
Let 
  \begin{equation*}
    \Phi^n_{\XXX} = \frac{1}{p^{n\,\codim\XXX}}\frob^n_\XXX
                       \colon \grot\proj(\XXX) \to \grot\proj(\XXX).
  \end{equation*}
According to~\cite[Theorem~19]{halvorsenF}, the endomorphism
$\Phi^n_\XXX$ is diagonalizable.

 In   Lemma~\ref{frobherzog}, we established that the functor
  $\rherzog^n$ is an endofunctor on $\inj(\XXX)$ which can be written as
\[
\rherzog^n(-)=(-)^\dagger\circ\lfrob^n \circ (-)^\dagger .
\]
Thus, $\rherzog^n$ is composed of functors that induce homomorphisms on
Grothendieck spaces, and hence it too
  induces a homomorphism $\grot\inj(\XXX)\to \grot\inj(\XXX)$. We denote
  this endomorphism on $\grot\inj(\XXX)$ by $\herzog_\XXX^n$. It is
  given for a complex $Y\in\inj(\XXX)$ by 
\[ 
\herzog^n_\XXX([Y]_{\inj(\XXX)})=[\rherzog^n(Y)]_{\inj(\XXX)}.
\]
Let 
  \begin{equation*}
    \Psi_\XXX^n = \frac{1}{p^{n\codim\XXX}}\herzog_\XXX^n
                     \colon \grot\inj(\XXX) \to \grot\inj(\XXX).
  \end{equation*}
  Theorem~\ref{injmaintheo} shows that
  $\Psi_\XXX^n$ also is a
  diagonalizable automorphism.

For complexes $X\in\proj(\XXX)$ and $Y\in\inj(\XXX)$ we shall write
$\Phi_\XXX^n(X)$ and $\Psi_\XXX^n(Y)$ instead of 
$\Phi_\XXX^n([X]_{\proj(\XXX)})$ and $\Psi_\XXX^n([Y]_{\inj(\XXX)})$,
respectively. The isomorphism in~\eqref{chiisomorphism} together with
Proposition~\ref{inducedhomomorphisms} shows that the Dutta
multiplicity $\chi_\infty(X,Y)$ and its two analogs $\xi_\infty(X,Y)$
and $\xi^\infty(X,Y)$ from Section~\ref{duttamultiplicity} can be
identified with elements in $\grot\der(\mmm)$ of the form
\begin{gather*}
\lim_{e\to\infty}(\Phi^e_\XXX(X)\otimes [Y]_{\der(\YYY)}), \quad
\lim_{e\to\infty}\hom([X]_{\der(\XXX)},\Psi^e_\YYY(Y))\text{\rlap{
    and}}\\
\lim_{e\to\infty}\hom(\Phi^e_\XXX(X),[Y]_{\der(\YYY)}) .
\end{gather*}

\section{Vanishing}

\subsection{Vanishing}\label{projvan}
  Let $\XXX$ be a specialization-closed subset of $\spec R$ and consider an element $\alpha$ in
  $\grot\proj(\XXX)$, an element $\beta$ in $\grot\inj(\XXX)$ and an
  element $\gamma$ in $\grot\der(\XXX)$. The \emph{dimensions} of
  $\alpha$, $\beta$ and $\gamma$ are defined as 
  \begin{align*}
   \dim\alpha &= \inf\strut \Big\{\dim(\supp X)\,\Big\vert\, \alpha=r[X]_{\proj(\XXX)} \text{ for some
   $r\in\QQ$ and $X\in\proj(\XXX)$}\Big\},\\
   \dim\beta &= \inf\strut \Big\{\dim(\supp Y)\,\Big\vert\, \alpha=s[Y]_{\proj(\XXX)} \text{ for some
   $s\in\QQ$ and $Y\in\inj(\XXX)$}\Big\}\text{\rlap{ and}}\\
   \dim\gamma  &= \inf \strut \Big\{\dim(\supp Z)\,\Big\vert\, \gamma=t[Z]_{\der(\XXX)} \text{ for some
   $t\in\QQ$ and $Z\in\der(\XXX)$}\Big\}.
  \end{align*}
  In particular, the dimension of an element in a Grothendieck space
  is $-\infty$ if and only if the element is trivial.
  We say that $\alpha$ \emph{satisfies vanishing} if
  \begin{equation*}
   \alpha\otimes\sigma=0 \quad\text{in $\grot\der(\mmm)$ for all $\sigma\in\grot\der(\comp{\XXX})$ with $\dim\sigma<\codim\XXX$},
  \end{equation*}
  and that $\alpha$ \emph{satisfies weak vanishing} if
  \begin{equation*}
   \overline{\alpha\otimes\tau}=0 \quad\text{in $\grot\der(\mmm)$ for all $\tau\in\grot\proj(\comp{\XXX})$ with $\dim\tau<\codim\XXX$}.
  \end{equation*}
Similarly, we say that $\beta$ \emph{satisfies vanishing} if 
  \begin{equation*}
   \hom(\sigma,\beta)=0 \quad\text{in $\grot\der(\mmm)$ for all $\sigma\in\grot\der(\comp{\XXX})$ with $\dim\sigma<\codim\XXX$},
  \end{equation*}
  and that $\beta$ \emph{satisfies weak vanishing} if
  \begin{equation*}
   \overline{\hom(\tau,\beta)}=0 \quad\text{in $\grot\der(\mmm)$ for all $\tau\in\grot\proj(\comp{\XXX})$ with $\dim\tau<\codim\XXX$}.
  \end{equation*}
  The \emph{vanishing dimension} of $\alpha$ and $\beta$ is defined as the numbers
  \begin{align*}
    \vandim\alpha &=\inf\strut\Big\{u\in\ZZ \,\Big\vert\,
    \begin{array}{l}\text{$\alpha\otimes \sigma=0$ for all
      $\sigma\in\grot\der(\comp\XXX)$}\\ \text{with $\dim\sigma < \codim\XXX-u$}\end{array}
    \Big\} \text{\rlap{ and}}\\
\vandim\beta &=\inf\strut\Big\{v\in\ZZ \,\Big\vert\,
    \begin{array}{l}\text{$\hom(\sigma,\beta)=0$ for all
      $\sigma\in\grot\der(\comp\XXX)$}\\ \text{with $\dim\sigma < \codim\XXX-v$}\end{array}
    \Big\} .
  \end{align*}
  In particular, the vanishing dimension of an element in a
  Grothendieck space is $-\infty$ if and only if the element is
  trivial, and  the vanishing dimension is less than or equal to $0$ if
  and only if the element satisfies vanishing.

\begin{rema}\label{vanexam}
If $X$ is a complex in $\proj(R)$ with $\XXX=\supp X$, then the
element $\alpha=[X]_{\proj(\XXX)}$ in $\grot\proj(\XXX)$ satisfies vanishing
  exactly when
  \begin{equation*}
    \chi(X,Y)=0 \quad\text{for all complexes $Y\in\der(\comp\XXX)$
                             with $\dim(\supp Y) < \codim\XXX$}, 
  \end{equation*}
  and $\alpha$ satisfies weak vanishing exactly when
  \begin{equation*}
    \chi(X,Y)=0 \quad\text{for all complexes $Y\in\proj(\comp\XXX)$
                             with $\dim(\supp Y)<\codim\XXX$}.
  \end{equation*}
  The vanishing dimension of $\alpha$ measures the extent to which
  vanishing fails to hold: the vanishing dimension of $\alpha$ is the
  infimum of integers $u$ such that 
  \begin{equation*}
    \chi(X,Y)=0 \quad\text{for all complexes $Y\in\der(\comp\XXX)$
                             with $\dim(\supp Y) < \codim\XXX-u$}. 
  \end{equation*}
It follows that the ring $R$ satisfies vanishing (or weak vanishing,
  respectively) as defined in~\ref{intmult}, if and only if all elements of $\grot\proj(\XXX)$ for all specialization-closed
  subsets $\XXX$ of $\spec R$ satisfy vanishing (or weak vanishing, respectively).

If $Y$ is a complex in $\inj(R)$ with $\XXX=\supp Y$, then the element
  $\beta = [Y]_{\inj(\XXX)}$ in $\grot\inj(\XXX)$
  satisfies vanishing exactly when
  \begin{equation*}
    \xi(X,Y)=0 \quad\text{for all complexes $X\in\der(\comp\XXX)$
                             with $\dim(\supp X) < \codim\XXX$}.
  \end{equation*}
  and $\beta$ satisfies weak vanishing exactly when
  \begin{equation*}
    \xi(X,Y)=0 \quad\text{for all complexes $X\in\proj(\comp\XXX)$
                             with $\dim(\supp X)<\codim\XXX$}.
  \end{equation*}
  The vanishing dimension of $\beta$ measures the extent to
  which vanishing of the Euler form fails to hold: the vanishing dimension
  of $\beta$ is the infimum of integers $v$ such that
  \begin{equation*}
    \xi(X,Y)=0 \quad\text{for all complexes $X\in\der(\comp\XXX)$
                              with $\dim(\supp X) < \codim\XXX - v$}. 
  \end{equation*}
  Because of the formulas in~\eqref{chixiformulas}, it follows that
  the ring $R$ satisfies vanishing (or weak vanishing, respectively)
  if and only all elements of $\grot\inj(\XXX)$ for all
  specialization-closed subsets $\XXX$ of $\spec R$ satisfy
  vanishing (or weak vanishing, respectively).  
\end{rema}

\begin{rema}\label{dimformulas}
For a specialization closed subset
$\XXX$ of $\spec R$ and elements $\alpha\in\grot\proj(\XXX)$,
$\beta\in\grot\inj(\XXX)$ and $\gamma\in\grot\der(\XXX)$,
we have the following formulas for dimension.
\begin{align*}
\dim\gamma &=\dim\gamma^\dagger, \\
\dim\alpha=\dim\alpha^\dagger &=\dim\alpha^*=\dim(D\otimes\alpha)
\text{\rlap{ and}}\\
\dim\beta=\dim\beta^\dagger &=\dim\beta^\star=\dim\hom(D,\beta).
 \end{align*}
These follow immediately from the fact that the dagger, star and Foxby
functors do not change supports of complexes. Further, we have the
following formulas for vanishing dimension. 
\begin{align*}
\vandim\alpha=\vandim\alpha^\dagger &=\vandim\alpha^* =\vandim(D\otimes\alpha)\rlap{\text{ and}}\\
\vandim\beta=\vandim\beta^\dagger &=\vandim\beta^\star=\vandim\hom(D,\beta) .
\end{align*}
These follow immediately from the above together with~\eqref{chixiformulas}.
\end{rema}

\begin{prop}\label{vanconditions}
  Let $\XXX$ be a specialization-closed subset of $\spec R$, let
  $\alpha\in\grot\proj(\XXX)$ and let $\beta\in\grot\inj(\XXX)$. Then the
  following hold.
  \begin{enumerate}
    \item \label{A1} If $\codim\XXX\leq 2$ then vanishing holds for all elements
      in $\grot\proj(\XXX)$ and $\grot\inj(\XXX)$. In particular, we always have
      \begin{equation*}
        \vandim\alpha, \vandim\beta \leq\max(0,\codim\XXX-2).
      \end{equation*}
    \item \label{A2} Let $\XXX'$ be a specialization-closed subset of $\spec R$
      with $\XXX\subseteq \XXX'$. Then
      \begin{equation*}
        \vandim\overline\alpha \leq \vandim\alpha
        -(\codim\XXX-\codim\XXX')\text{\rlap{ and}}
      \end{equation*}
for $\overline\alpha\in\grot\proj(\XXX')$. For any given $s$ in the
range $0\leq s\leq \vandim\alpha$, we can always find 
an $\XXX'$ with $s=\codim\XXX-\codim\XXX'$ such that the above
inequality becomes an equality. Likewise,
\[
        \vandim\overline\beta \leq \vandim\beta
        -(\codim\XXX-\codim\XXX')
\]
for $\overline\beta\in\grot\inj(\XXX')$, and for any given $s$ in the
range $0\leq s\leq \vandim\beta$, we can always find an $\XXX'$ with
$s=\codim\XXX-\codim\XXX'$ such that the above inequality becomes an
equality. 
\item \label{A3}The element $\alpha$ satisfies weak
      vanishing if and only if, for all specialization-closed subsets
      $\XXX'$ with $\XXX\subseteq\XXX'$ and $\codim\XXX'=\codim\XXX-1$,
      \begin{equation*}
        \overline\alpha = 0 \quad\text{as an element of}\quad \grot\der(\XXX').
      \end{equation*} 
Similarly, the element $\beta$ satisfies weak vanishing if and only
if, for all specialization-closed subsets $\XXX'$ with
$\XXX\subseteq\XXX'$ and $\codim\XXX'=\codim\XXX-1$, 
      \begin{equation*}
        \overline\beta = 0 \quad\text{as an element of}\quad \grot\der(\XXX').
      \end{equation*} 
  \end{enumerate}
\end{prop}
\begin{proof}
Because of Proposition~\ref{formulas} and the formulas in Remark~\ref{dimformulas}, it suffices to
consider the statements for $\alpha$ and $\grot\proj(\XXX)$. But the
in this case, \eqref{A1} and \eqref{A2} are already contained
in~\cite[Example~6 and Remark~7]{halvorsenF}, and \eqref{A3} follows by
considerations similar to those proving ($ii$) in~\cite[Remark~7]{halvorsenF}. 
\end{proof}

The following two propositions present conditions that are equivalent to having a
certain vanishing dimension for elements of the Grothendieck space $\grot\inj(\XXX)$. 
There are similar results for elements of the Grothendieck space
$\grot\proj(\XXX)$; see~\cite[Proposition~23 and~24]{halvorsenF}.

\begin{prop}\label{injvandim0}
  Let $\XXX$ be a specialization-closed subset of $\spec R$, and let $\beta \in
  \grot\inj(\XXX)$. Then the following conditions are equivalent.
  \begin{enumerate}
    \item \label{B1} $\vandim\beta \leq 0$.
    \item \label{B2} $\hom(\gamma,\beta) = 0$ for all
      $\gamma\in\grot\der(\comp{\XXX})$ with $\dim\gamma <\codim\XXX$.
    \item \label{B3} $\overline\beta=0$ in $\grot\inj(\XXX')$ for any specialization-closed
    subset $\XXX'$ of $\spec R$ with $\XXX\subseteq\XXX'$ and
    $\codim\XXX' <\codim\XXX$.
    \item \label{B4} $\overline\beta=0$ in $\grot\inj(\XXX')$ for any specialization-closed subset
    $\XXX'$ of $\spec R$ with $\XXX\subseteq\XXX'$ and $\codim\XXX' =
    \codim\XXX -1$.
  \end{enumerate}
\end{prop}
\begin{proof}
  By definition \eqref{B1} is equivalent to \eqref{B2}, and
  Proposition~\ref{grothendieckobs} in conjunction with
  Remark~\ref{vanexam} shows that \eqref{B1} implies \eqref{B3}. Clearly
  \eqref{B3} is stronger than \eqref{B4}, and \eqref{B4} in conjunction with
  Proposition~\ref{vanconditions} implies \eqref{B2}.
\end{proof}

\begin{prop}\label{injvandimv}
  Let $\XXX$ be a specialization-closed subset of $\spec R$, let $\beta \in
  \grot\inj(\XXX)$, and let $u$ be a non-negative integer. Then the
  following conditions are equivalent.
  \begin{enumerate}
  \item $\vandim\beta \leq v$.
  \item $\hom(\gamma,\beta) = 0$ for all
    $\gamma\in\grot\der(\comp{\XXX})$ with $\dim\gamma <\codim\XXX-v$.
  \item $\overline\beta=0$ in of $\grot\inj(\XXX')$ for any specialization-closed
    subset $\XXX'$ of $\spec R$ with $\XXX\subseteq\XXX'$ and
    $\codim\XXX' <\codim\XXX-u$.
  \item $\overline\beta=0$ in $\grot\inj(\XXX')$ for any specialization-closed
    subset $\XXX'$ of $\spec R$ with $\XXX\subseteq\XXX'$ and
    $\codim\XXX' = \codim\XXX -v-1$.
  \end{enumerate} 
\end{prop}

\begin{proof}
  The structure of the proof is similar to that of
  Proposition~\eqref{injvandim0}.
\end{proof}

\section{Grothendieck spaces in prime characteristic}
According to~\cite[Theorem~19]{halvorsenF} the endomorphism
$\Phi_\XXX$ on $\grot\proj(\XXX)$ is diagonalizable; the precise
statement is recalled in the next theorem. This section establishes
that the endomorphism $\Psi_\XXX$ on $\grot\inj(\XXX)$ is also
diagonalizable; the precise statement is Theorem~\ref{injmaintheo} below.

\begin{theo}\label{diagonalizingthefrobenius}
Assume that $R$ is complete of prime characteristic $p$ and with perfect
  residue field, and let $\XXX$ be a specialization-closed subset of $\spec R$. If
  $\alpha$ is an element in $\grot\proj(\XXX)$ and $u$ is a non-negative
  integer with $u \geq \vandim\alpha$, then
  \begin{equation*}
    (p^u\Phi_\XXX - \id)\circ \cdots \circ 
    (p\Phi_\XXX - \id) \circ (\Phi_\XXX - \id)(\alpha) =0 ,
  \end{equation*}
  and there exists a unique decomposition
  \begin{equation*}
    \alpha = \alpha^{(0)} + \cdots + \alpha^{(u)} 
  \end{equation*}
  in which each $\alpha^{(i)}$ is either zero or an eigenvector for
  $\Phi_\XXX$ with eigenvalue $p^{-i}$.  The elements
$\alpha^{(i)}$ can be computed according to the formula 
  \begin{equation*}
    \begin{pmatrix} \alpha^{(0)} \\ \vdots \\ \alpha^{(u)} \end{pmatrix} =
    \begin{pmatrix} 1 & 1   & \cdots & 1 \\ 
                    1 & p^{-1} & \cdots & p^{-u} \\
                    \vdots & \vdots & \ddots & \vdots \\
                    1 & p^{-u} & \cdots & p^{-u^2}\end{pmatrix}^{\!\!\!-1} 
    \!\!\!\begin{pmatrix} \alpha \\ \Phi_\XXX(\alpha) \\ \vdots \\
      \Phi_\XXX^u(\alpha) \end{pmatrix} ,
  \end{equation*}
and may also be recursively obtained as
\[
\alpha^{(0)}=\lim_{e\to\infty}\Phi_\XXX^e(\alpha) \quad\text{and}\quad \alpha^{(i)}=
\lim_{e\to\infty} p^{ie}\Phi^e_\XXX(\alpha-(\alpha^{(0)}+\cdots
+\alpha^{(i-1)})) .
\]
\end{theo}

\begin{theo}\label{injmaintheo}
  Assume that $R$ is complete of prime characteristic $p$ and with perfect
  residue field, and let $\XXX$ be a specialization-closed subset of $\spec R$. If
  $\beta$ is an element in $\grot\inj(\XXX)$ and $v$ is a non-negative
  integer with $v \geq \vandim\beta$, then
  \begin{equation*}
    (p^v\Psi_\XXX-\id)\circ \cdots \circ (p\Psi_\XXX-\id) \circ (\Psi_\XXX-\id)(\beta) =0,
  \end{equation*}
  and there exists a unique decomposition
  \begin{equation*}
    \beta=\beta^{(0)}+\cdots + \beta^{(v)} ,
  \end{equation*}
  in which each $\beta^{(i)}$ is either zero or an eigenvector for
  $\Psi_\XXX$ with eigenvalue $p^{-i}$.
The elements $\beta^{(i)}$ can be computed according to the formula 
 \begin{equation}\label{injgeneralformula}
    \begin{pmatrix} \beta^{(0)} \\ \vdots \\ \beta^{(u)} \end{pmatrix} =
    \begin{pmatrix} 1 & 1  & \cdots & 1 \\ 
      1 & p^{-1} & \cdots & p^{-v} \\
      \vdots & \vdots & \ddots & \vdots \\
      1 & p^{-v} & \cdots & p^{- v^2}\end{pmatrix}^{\!\!\!-1}
    \!\!\!\begin{pmatrix} \beta \\ \Psi_\XXX(\beta) \\ \vdots \\
    \Psi_\XXX^v(\beta) \end{pmatrix},
  \end{equation} 
and may also be recursively obtained as
\[
\beta^{(0)}=\lim_{e\to\infty}\Psi_\XXX^e(\beta) \quad\text{and}\quad \beta^{(i)}=
\lim_{e\to\infty} p^{ie}\Psi^e_\XXX(\beta-(\beta^{(0)}+\cdots +\beta^{(i-1)})).
\]
\end{theo}
\begin{proof}
On the injective Grothendieck space $\grot\inj(\XXX)$, the identities
described in Lemma~\ref{frobherzog} imply that we have the
following commutative diagram.
\[
\xymatrix@C+40pt@R+10pt{
  {\grot\proj(\XXX)} \ar[r]^-{\Phi_\XXX^n}_-{\cong} 
  \ar[d]_-{(-)^{\dagger}}^-{\cong} & {\grot\proj(\XXX)} \ar@{<-}[d]^-{(-)^{\dagger}}_-{\cong} \\
  {\grot\inj(\XXX)} \ar[r]^-{\Psi_\XXX^n} & {\grot\inj(\XXX)}
    }
\]
In particular,
  \begin{equation*}
    \Psi_\XXX(-) = (-)^{\dagger} \circ \Phi_\XXX \circ (-)^{\dagger}.
  \end{equation*}
By Remark~\ref{dimformulas}, we have $v\geq
\vandim\beta=\vandim\beta^\dagger$, so Theorem~\ref{diagonalizingthefrobenius}
and the above identity yields that 
\begin{align}\label{psiformula}
  (p^v\Psi_\XXX - \id)\circ \cdots \circ (p\Psi_\XXX - \id) \circ (\Psi_\XXX -\id)(\beta) 
  &= 0.
\end{align} 
Applying $\Psi_\XXX^{e-v}$ to~\eqref{psiformula} results in a
recursive formula to compute $\Psi_\XXX^{e+1}(\beta)$ from 
$\Psi_\XXX^e(\beta),\dots ,\Psi_\XXX^{e-v}(\beta)$. The characteristic
polynomial for the recursion is
\[
(p^vx-1)\cdots (px-1)(x-1),
\]
which has $v+1$ distinct roots $1,p^{-1},\dots ,p^{-v}$. Consequently,
there exist elements $\beta^{(0)},\dots ,\beta^{(v)}$ such that
\[
\Psi_\XXX^e(\beta) = \beta^{(0)}+p^{-e}\beta^{(1)} +\cdots +p^{-ve}\beta^{(v)},
\]
where each $\beta^{(i)}$ is an eigenvector for $\Psi_\XXX$ with
eigenvalue $p^{-i}$. Setting $e=0$ obtains the decomposition
$\beta=\beta^{(0)}+\cdots +\beta^{(v)}$, and solving the system
of linear equations obtained by setting $e=0,\dots ,v$
shows~\eqref{injgeneralformula}; observe that the matrix is the Vandermonde 
matrix on $1,p^{-1},\dots ,p^{-v}$, which is invertible. The formula
also immediately shows that
$\lim_{e\to\infty}\Psi^e(\beta)=\beta^{(0)}$ and that
\begin{align*}
\lim_{e\to\infty}p^{ie}\Psi_\XXX^e(\beta-(\beta^{(0)}+\cdots
+\beta^{(i-1)})) &=
\lim_{e\to\infty}p^{ie}\Psi_\XXX^e(\beta^{(i)}+\cdots +\beta^{(v)}) \\
&= \lim_{e\to\infty}(\beta^{(i)}+\cdots + p^{-(v-i)e}\beta^{(v)}) \\
 &= \beta^{(i)}.
\end{align*}
This concludes the argument.
\end{proof}

\begin{prop}\label{D2}
  Assume that $R$ is a complete ring of prime characteristic $p$ and
  with perfect residue field, and let $\XXX$ be a specialization-closed subset of $\spec R$. Consider the
  following diagram.
\begin{equation*}
\xymatrix@R+25pt@C=-12pt{
{\phantom{|}} \ar@(dl,ul)[]^{\Phi_\XXX}  & {\grot\proj(\XXX)}  
\ar@<.6ex>[rr]^-{(-)^\dagger}
\ar@/^20pt/^-{D\otimes -}[rr] & {\hspace{150pt}} & {\,\grot\inj(\XXX)} 
\ar@<.6ex>[ll]^-{(-)^\dagger} \ar@/^20pt/[ll]^-{\hom(D,-)} & {\phantom{|}} \ar@(dr,ur)[]_{\Psi_\XXX}
}
\end{equation*}
 For the Grothendieck space $\grot\inj(\XXX)$, we have the following identities.
\begin{align*} 
%\begin{split}
\Psi_\XXX(-) = \Phi_\XXX(-^\dagger)^\dagger &= D\otimes \Phi_\XXX(\hom(D,-)) .\\
 \quad (-)^{(i)} = (-)^{\dagger (i)\dagger} &= D\otimes 
 (\hom(D,-)^{(i)}).
\end{align*}
\end{prop}
\begin{proof}
The formulas in the first line are an immediate consequence of
Lemma~\ref{frobherzog}. Let $\beta$ be an element in 
$\grot\inj(\XXX)$. Using the decomposition in $\grot\proj(\XXX)$ from 
Theorem~\ref{diagonalizingthefrobenius}, we can write 
\[
\beta=\beta^{\dagger\dagger} = \beta^{\dagger(0)\dagger} +\cdots +\beta^{\dagger(v)\dagger},
\]
and since
\[
\Psi_\XXX(\beta^{\dagger(i)\dagger}) =
\Phi_\XXX(\beta^{\dagger(i)})^\dagger =
p^{-i}\beta^{\dagger(i)\dagger} ,
\]
we learn from the uniqueness of the decomposition that
$\beta^{(i)}=\beta^{\dagger(i)\dagger}$. This proves the first
equality in the second line. The last equality follows by similar considerations.
\end{proof}

\begin{rema}\label{duttaeulerformulas}
In~\cite[Remark~21]{halvorsenF} it is established that the Dutta
multiplicity is computable. Employing
Theorems~\ref{diagonalizingthefrobenius} and~\ref{injmaintheo}
together with the fact from Proposition~\ref{inducedhomomorphisms}
that the induced Hom-homomorphism on Grothendieck spaces is continuous
in both variables, it follows, as will be shown below,  that the two analogs of Dutta
multiplicity are also computable.

Let $X$ and $Y$ be finite complexes. Set $\XXX=\supp X$ and
$\YYY=\supp Y$, and assume that $\XXX\cap\YYY=\{\mmm\}$ and $\dim
\XXX+\dim \YYY\leq\dim R$. Then, in the case where $Y$ is in
$\inj(R)$, the multiplicity $\xi_\infty(X,Y)$ can be identified
via~\eqref{chiisomorphism} with the element  
\begin{align*}
\lim_{e\to\infty}\hom([X]_{\der(\comp\YYY)},\Psi^e_\YYY(Y)) &=
\hom([X]_{\der(\comp\YYY)},\lim_{e\to\infty}\Psi^e_\YYY(Y)) \\
&= \hom([X]_{\der(\comp\YYY)},[Y]_{\inj(\YYY)}^{(0)}) ,
\end{align*}
whereas, in the case where $X$ is in $\proj(R)$, the multiplicity 
$\xi^\infty(X,Y)$ can be identified via~\eqref{chiisomorphism} with the element
\begin{align*}
\lim_{e\to\infty}\hom(\Phi^e_\XXX(X),[Y]_{\der(\comp\XXX)}) &=
\hom(\lim_{e\to\infty}\Phi^e_\XXX(X),[Y]_{\der(\comp\XXX)}) \\
 &= \hom([X]^{(0)}_{\proj(\XXX)},[Y]_{\der(\comp\XXX)}) .
\end{align*}
The formulas in Theorems~\ref{diagonalizingthefrobenius}
and~\ref{injmaintheo} now yield formulas for $\xi_\infty(X,Y)$ and 
$\xi^\infty(X,Y)$ as presented in the corollary below.
\end{rema}
\begin{coro}
  Assume that $R$ is a complete ring of prime characteristic $p$ and
  with perfect residue field. Let $X$ and $Y$ be finite complexes with
  \[
\supp X\cap\supp Y=\{\mmm\}\quad\text{and}\quad \dim(\supp
X)+\dim(\supp Y)\leq\dim R.
\]
When $Y\in\inj(R)$, letting $v$ denote the
vanishing dimension of $[Y]_{\inj(\supp Y)}$ and setting $t=\codim(\supp Y)$, we have
\[
\xi_\infty(X,Y) = 
\begin{pmatrix} 1 & 0 & \cdots &  0 \end{pmatrix}
 \begin{pmatrix} 1 & 1  & \cdots & 1 \\ 
 p^t & p^{t-1} & \cdots & p^{t-v} \\
 \vdots & \vdots & \ddots & \vdots \\
 p^{vt} & p^{v(t-1)} & \cdots &
 p^{v(t-v)}\end{pmatrix}^{\!\!\!-1}\!\!\! \begin{pmatrix} \xi(X,Y) \\ \xi(X,\rherzog(Y)) \\ \vdots \\
   \xi(X,\rherzog^v(Y)) \end{pmatrix} ,
\]
and when $X\in\proj(R)$, letting
$u$ denote the vanishing dimension of $[X]_{\proj(\supp X)}$ and setting $s=\codim(\supp X)$, we have
\[
\xi^\infty(X,Y) = 
\begin{pmatrix} 1 & 0 & \cdots &  0 \end{pmatrix}
 \begin{pmatrix} 1 & 1  & \cdots & 1 \\ 
 p^s & p^{s-1} & \cdots & p^{s-u} \\
 \vdots & \vdots & \ddots & \vdots \\
 p^{us} & p^{u(s-1)} & \cdots &
 p^{u(s-u)}\end{pmatrix}^{\!\!\!-1}\!\!\! \begin{pmatrix} \xi(X,Y) \\ \xi(\lfrob(X),Y) \\ \vdots \\
   \xi(\lfrob^u(X),Y) \end{pmatrix} .
\]
Thus, it is possible to calculate $\xi_\infty(X,Y)$ and $\xi^\infty(X,Y)$ as
$\QQ$--linear combinations of ordinary Euler forms; in particular,
they are rational numbers.
\end{coro}
Note that the above corollary also can be obtained directly
from~\cite[Remark~21]{halvorsenF} by employing Lemma~\ref{frobherzog}
and the formulas in~\eqref{chixiformulas}.

\begin{rema}\label{injdecompositioninclusion}
  Let $\XXX$ and $\XXX'$ be specialization-closed subsets of $\spec R$ such that
  $\XXX\subseteq\XXX'$. Set $s = \codim\XXX - \codim\XXX'$ and
  consider the inclusion homomorphism
  \begin{equation*}
    \overline{(-)} \colon \grot\inj(\XXX)\to \grot\inj(\XXX').
  \end{equation*}
  Pick an element $\beta\in\grot\inj(\XXX)$, and apply the
  convention that $\beta^{(t)}=0$ for all negative integers $t$.  It
  follows immediately that
  \begin{equation*}
    \Psi_{\XXX'}(\overline \beta) = p^s\overline{\Psi_\XXX(\beta)} ,
  \end{equation*}
  and employing Theorem~\ref{injmaintheo} we obtain the
  identity $\overline{\beta^{(i)}} = \overline{\beta}^{(i - s)}$.
  The situation may be visualized as follows
  \begin{equation*}
  \xymatrix@C=0pt{
    {\grot\inj(\XXX) \ni} & \beta \ar@{->}[d] & = &
    \beta^{(0)}& + & \cdots & + &
    \beta^{(s)}\ar@{->}[lllld] & + &
    \beta^{(s+1)}\ar@{->}[lllld]  & + & \cdots & + & \beta^{(v)}\ar@{->}[lllld] \\
    {\grot\inj(\XXX') \ni} & \overline\beta &=&
    \overline\beta^{(0)}& + & \overline\beta^{(1)} & + & \cdots & + &
    \overline\beta^{(v-s)}. \\
    }
   \end{equation*}
There are similar results for elements $\alpha\in\grot\proj(\XXX)$;
see~\cite[Remark~20]{halvorsenF}.
\end{rema}

  The following two propositions characterize vanishing dimension for
  elements of the Grothendieck space 
  $\grot\inj(\XXX)$. They should be read in parallel with
  Propositions~\ref{injvandim0} and~\ref{injvandimv}. There are
  similar results for the Grothendieck space $\grot\proj(\XXX)$;
  see~\cite[Proposition~23 and~24]{halvorsenF}.

\begin{prop}\label{injconditionsdim0}
  Assume that $R$ is complete of prime characteristic $p$ and with
  perfect residue field. Let $\XXX$ be a specialization-closed subset of $\spec R$
  and let $\beta \in \grot\inj(\XXX)$. The following are equivalent.
  \begin{enumerate}
   \item \label{C1} $\beta$ satisfies vanishing.
   \item \label{C2} $\vandim\beta\leq 0$.  
   \item \label{C3} $\beta=\beta^{(0)}$.
   \item \label{C4} $\beta = \Psi_\XXX(\beta)$.
   \item \label{C5} $\beta = \Psi_\XXX^e(\beta)$ for some $e\in\NN$.
   \item \label{C6} $\beta = \lim_{e\to\infty}\Psi_\XXX^e(\beta)$.
\end{enumerate}
\end{prop}
\begin{proof}
  By definition \eqref{C1} and \eqref{C2} are equivalent, and from
  Theorem~\ref{injmaintheo} it follows that \eqref{C2}
  implies~\eqref{C3}. Moreover, Theorem~\ref{injmaintheo} shows that
  the four conditions \eqref{C3}--\eqref{C6} are equivalent. Finally,
  condition~\eqref{C3} implies condition~\eqref{C1} through a reference to
  Remark~\ref{injdecompositioninclusion} and Proposition~\ref{injvandim0}.
\end{proof}

\begin{prop}\label{injconditionsdimu}
  Assume that $R$ is complete of prime characteristic $p$ and with
  perfect residue field. Let $\XXX$ be a specialization-closed subset of $\spec R$,
  let $\beta \in \grot\inj(\XXX)$ and let $v$ be a non-negative
  integer. The following are equivalent.
  \begin{enumerate}
    \item \label{CC1} $\vandim\beta \leq v$.
    \item \label{CC2} $\beta=\beta^{(0)}+\cdots +\beta^{(v)}$.
    \item \label{CC3} $(p^v\Psi_\XXX -\id)\circ \cdots \circ (p\Psi_\XXX-\id)\circ
    (\Psi_\XXX -\id)(\beta)=0$.
  \end{enumerate}
\end{prop}

\begin{proof}
  From Theorem~\ref{injmaintheo} it follows that~\eqref{CC1}
  implies~\eqref{CC2} which is equivalent to~\eqref{CC3}. Since
  $\beta^{(i)}\neq 0$ implies  $\vandim\beta^{(i)} = i$ by
  Remark~\ref{injdecompositioninclusion}  and
  Proposition~\ref{injvandimv}, it follows that~\eqref{CC2} implies~\eqref{CC1}.
\end{proof}

\begin{prop}\label{01}
  Assume that $R$ is complete of prime characteristic $p$ and with
  perfect residue field.  Let $\XXX$ and $\YYY$ be specialization-closed subsets of
  $\spec R$ such that $\XXX\cap\YYY=\{\mmm\}$ and $\dim\XXX+\dim\YYY\leq\dim R$, and set $e = \dim R -
  (\dim\XXX + \dim\YYY)$. If $(\sigma,\tau)$ is a pair of elements
  from
\[
\grot\proj(\XXX)\times \grot\proj(\YYY),\quad \grot\proj(\XXX)\times
\grot\inj(\YYY) \quad\text{or}\quad \grot\inj(\XXX)\times \grot\proj(\YYY),
\]
so that $\sigma\otimes\tau$ is a well-defined element of
$\grot\proj(\mmm)$ or $\grot\inj(\mmm)$, then
  \begin{equation}\label{tensorcomponent}
    (\sigma \otimes \tau)^{(i)} =
    \sum_{m + n = i + e}\sigma^{(m)}\otimes\tau^{(n)}.
  \end{equation}
If instead $(\sigma,\tau)$ is a pair of elements from
\[
\grot\proj(\XXX)\times \grot\proj(\YYY),\quad \grot\proj(\XXX)\times
\grot\inj(\YYY) \quad\text{or}\quad \grot\inj(\XXX)\times \grot\inj(\YYY),
\]
so that $\hom(\sigma,\tau)$ is a well-defined element of
$\grot\proj(\mmm)$ or $\grot\inj(\mmm)$, then
  \begin{equation*}
    \hom(\sigma,\tau)^{(i)} = 
    \sum_{m + n = i + e}\hom(\sigma^{(m)},\tau^{(n)}).
  \end{equation*}
\end{prop}

\begin{proof}
We will verify that~\eqref{tensorcomponent} holds in the case where
$(\sigma,\tau)$ is pair of elements from
$\grot\proj(\XXX)\times\grot\proj(\YYY)$. The verification of the remaining
statements is similar.

It suffices to argue that the element
  \begin{equation*}
    \alpha =  \sum_{m + n = i + e}\sigma^{(m)}\otimes\tau^{(n)} \in \grot\proj(\mmm)
  \end{equation*}
  is an eigenvector for $\Phi_\mmm = \Phi_{\XXX\cap\YYY}$ with
  eigenvalue $p^{-i}$. We compute
  \begin{align*}
    \Phi_\mmm(\alpha) 
    &= 
    \sum_{m + n = i + e}p^{-\dim R}\frob_\mmm(\sigma^{(m)}\otimes\tau^{(n)})\\
    &= 
    p^{-\dim R}\sum_{m + n = i + e}\frob_\XXX(\sigma^{(m)})\otimes \frob_\YYY(\tau^{(n)})\\
    &=
    p^{-\dim R}\sum_{m + n = i + e}p^{\codim\XXX}\Phi_\XXX(\sigma^{(m)})
                                 \otimes p^{\codim\YYY}\Phi_\YYY(\tau^{(n)})\\
    &=
    p^{-i}\sum_{m + n = i + e}\sigma^{(m)}\otimes\tau^{(n)} = p^{-i}\alpha.
  \end{align*}
  Here, all equalities but the second are propelled only by
  definitions. The second equality follows from Proposition~\ref{square}.
\end{proof}

In~\cite{halvorsenF}, the concept of ``numerical vanishing'' is
introduced for elements $\alpha$ of the Grothendieck space
$\grot\proj(\XXX)$. We here repeat the definition and extend it to
elements $\beta$ in the Grothendieck space $\grot\inj(\XXX)$.
  
\begin{defi}
Assume that $R$ is complete of prime characteristic $p$ and with
perfect residue field, and let $\XXX$ be a specialization-closed
subset of $\spec R$. An element $\alpha\in\grot\proj(\XXX)$ is said
to \emph{satisfy numerical vanishing} if the images in
$\grot\der(\mmm)$ of $\alpha$ and $\alpha^{(0)}$ coincide. An element
$\beta\in\grot\inj(\XXX)$ is said to \emph{satisfy numerical
  vanishing} if the images in  $\grot\der(\mmm)$ of $\beta$ and
$\beta^{(0)}$ coincide. The ring $R$ is said to \emph{satisfy
  numerical vanishing} if all elements of the Grothendieck space
$\grot\proj(\XXX)$ satisfy numerical vanishing for all specialization-closed
subsets $\XXX$ of $\spec R$.
\end{defi}

\begin{rema}\label{injnumvan}
$R$ satisfies numerical vanishing precisely when all elements of the Grothendieck space
$\grot\inj(\XXX)$ satisfy numerical vanishing for all specialization-closed
subsets $\XXX$ of $\spec R$. To see
this, simply note that, by Proposition~\ref{D2}, the element $\beta$ in $\grot\inj(\XXX)$
satisfies numerical vanishing if and only if the 
corresponding element $\beta^\dagger$ in $\grot\proj(\XXX)$ does. This
observation allows us in the following proposition to present an
injective version of~\cite[Remark~28]{halvorsenF}.
\end{rema}

\begin{prop}\label{suffcondforweakvan}
Assume that $R$ is complete of prime characteristic $p$ and with
perfect residue field.
A necessary and sufficient condition for $R$ to satisfy numerical
vanishing is that all element of $\grot\inj(\mmm)$ satisfy numerical
vanishing: that is, that
\[
\chi(\rherzog(Y)) = p^{\dim R}\chi(Y)
\]
for all complexes $Y\in\inj(\mmm)$. If $R$ is
Cohen--Macaulay, $\grot\inj(\mmm)$ is generated by modules, and hence
a necessary and sufficient condition for $R$ to 
satisfy numerical vanishing is that 
\[
\length \herzog(N) = p^{\dim R}\length N
\]
for all modules $N$ with finite length and finite injective
dimension.
\end{prop}
\begin{proof}
The proposition follows immediately from~\cite[Remark~28]{halvorsenF} by
applying the dagger duality isomorphism between $\grot\proj(\mmm)$ and
$\grot\inj(\mmm)$ and by noting that 
 $(-)^\dagger$ takes a module in $\proj(\mmm)$ to a module in
 $\inj(\mmm)$ and vice versa.
\end{proof}

\section{Self-duality}
  Let $\XXX$ be a specialization-closed subset of $\spec R$, and let $K$ be a Koszul
  complex in $\proj(\XXX)$ on $\codim\XXX$ elements. It is a well-know
  fact that Koszul complexes are ``self-dual'' in the sense that $K
  \simeq \shift^{\codim\XXX}K^*$. In particular, for the element 
  $\alpha=[K]_{\proj(\XXX)}$ in $\grot\proj(\XXX)$, we have
  \begin{equation*}
    \alpha= [\shift^{\codim\XXX}K^*]_{\proj(\XXX)} =
    (-1)^{\codim\XXX}[K^*]_{\proj(\XXX)}=(-1)^{\codim\XXX}\alpha^*.
  \end{equation*}
  Proposition~\ref{selfdualityforvanishing} below
  shows that this feature is displayed for all elements that satisfy
  vanishing. 

\begin{defi}
  Let $\XXX$ be a specialization-closed subset of $\spec R$ and consider an element
  $\alpha\in\grot\proj(\XXX)$ and an element
  $\beta\in\grot\inj(\XXX)$. If 
  \begin{equation*}
    \alpha = (-1)^{\codim\XXX}\alpha^*,
  \end{equation*}
  we say that $\alpha$ is \emph{self-dual}, and
  if all elements in $\grot\proj(\XXX)$ for all specialization-closed subsets $\XXX$
  of $\spec R$ are self-dual, we say that
  $R$ \emph{satisfies self-duality}. Moreover, if the above equality
  holds after an application of the inclusion homomorphism
  $\grot\proj(\XXX)\to\grot\der(\XXX)$ so that 
\[
\overline\alpha =  (-1)^{\codim\XXX}\overline{\alpha^*}
\]
 in $\grot\der(\XXX)$, we say that $\alpha$ is \emph{numerically
    self-dual}, and if all elements in $\grot\proj(\XXX)$ for all specialization-closed subsets $\XXX$
  of $\spec R$ are numerically self-dual, we say that $R$
  \emph{satisfies numerical self-duality}.

Similarly, if 
  \begin{equation*}
    \beta = (-1)^{\codim\XXX}\beta^\star,
  \end{equation*}
  we say that $\beta$ is
  \emph{self-dual}, and if the above equality
  holds after an application of the inclusion homomorphism
  $\grot\inj(\XXX)\to\grot\der(\XXX)$ so that 
\[
\overline\beta =   (-1)^{\codim\XXX}\overline{\beta^\star}
\]
 in $\grot\der(\XXX)$, we say that $\beta$ is \emph{numerically
    self-dual}.
\end{defi}

\begin{rema}
The commutativity of the star and dagger functors shows that an
element $\beta\in\grot\inj(\XXX)$ is self-dual if and only if the
corresponding element $\beta^\dagger\in\grot\proj(\XXX)$ is
self-dual. Thus, $R$ satisfies self-duality if and only if all
elements in $\grot\inj(\XXX)$ for all specialization-closed subsets $\XXX$ of $\spec
R$ are self-dual. A similar remark holds for numerical self-duality.
\end{rema}

\begin{prop}\label{D3}
  Let $\XXX$ be a specialization-closed subset of $\spec R$, let
  $\alpha\in\grot\proj(\XXX)$ and let $\beta\in\grot\inj(\XXX)$. Then,
\[
    \vandim\alpha^* = \vandim\alpha \quad\text{and}\quad
  \vandim\beta^\star = \vandim\beta .
\]
If, in addition, $R$ is complete of prime characteristic $p$ and with perfect residue
  field, we have
\[
    \Phi_\XXX(\alpha^*) =\Phi_\XXX(\alpha)^*  \quad\text{and}\quad
    \Psi_\XXX(\beta^{\star}) =\Psi_\XXX(\beta)^{\star} .
\]
In particular, for all integers $i$, we have
\[
    (\alpha^*)^{(i)} =(\alpha^{(i)})^* \quad\text{and}\quad
    (\beta^\star)^{(i)} = (\beta^{(i)})^\star.
\]
\end{prop}

\begin{proof} 
The
formulas for vanishing dimension follow from~\eqref{chixiformulas},
since the dagger functor 
does not change the dimension of a complex. The second pair of
formulas follow immediately from the commutativity of the star and
Frobenius functors; see~\ref{endofunctors}. Thus, it follows that
\[
\Phi_\XXX(\alpha^{(i)*})=\Phi_\XXX(\alpha^{(i)})^*=p^{-i}\alpha^{(i)*}.
\]
That is to say, $\alpha^{(i)*}$ is an eigenvector for $\Phi_\XXX$ with
eigenvalue $p^{-i}$. Setting $u=\vandim\alpha$, the decomposition
\[
\alpha^*=(\alpha^{(0)}+\cdots +\alpha^{(u)})^*= \alpha^{(0)*}+\cdots +\alpha^{(u)*}
\]
now shows that $\alpha^{*(i)}=\alpha^{(i)*}$. A similar argument applies
for $\beta$.
\end{proof}

\begin{prop}\label{selfdualityforvanishing}
Let $\XXX$ be a specialization-closed subset of $\spec R$. If an
element $\alpha\in\grot\proj(\XXX)$ satisfies vanishing, then $\alpha$
is self-dual, and if an element $\beta\in\grot\inj(\XXX)$ satisfies
vanishing, then $\beta$ is self-dual. 
Moreover, $R$ 
  satisfies vanishing if and only if $R$ satisfies self-duality, and if
  $R$ satisfies numerical self-duality, then $R$ satisfies weak vanishing.
\end{prop}

\begin{proof}
We shall prove that, if $\alpha$ satisfies vanishing, then $\alpha$ is
self-dual. The corresponding statement for $\beta$ follows from dagger
duality, since $\beta$ is self-dual exactly when $\beta^\dagger$ is
and satisfies vanishing exactly when $\beta^\dagger$ does. 

  By Proposition~\ref{grothendieckobs}, it suffices to assume that
  $\alpha=[X]_{\proj(\XXX)}$ for a complex $X$ from $\proj(\XXX)$. We
  are required to establish the identity
  \begin{equation}\label{goal}
    \chi(X^*\ltensor_R -) = (-1)^{\codim\XXX} \chi(X\ltensor_R -)
  \end{equation}
  viewed as metafunctions on $\der(\comp{\XXX})$. First, we translate this
  question into showing that, if $R$ is a domain and $\XXX$
  equals $\mmm$, then
  \begin{equation*}
     \chi(X^*)=(-1)^{\dim R}\chi(X)
  \end{equation*}
for all complexes $X$ in $\proj(\mmm)$ such that $[X]_{\proj(\mmm)}$
satisfies vanishing.

  \medskip
  \noindent
  $\mathbf{1^{\circ}}$ By assumption, $\alpha$ satisfies vanishing, and
  Proposition~\ref{D3} implies that so does $\alpha^*$.
  From
  Proposition~\ref{grothendieckobs} we see that, in order to show \eqref{goal},
  it suffices to test with modules of the form $R/\ppp$ for prime
  ideals $\ppp$ from $\comp{\XXX}$ with $\dim R/\ppp = \codim \XXX$.
  Consider the following computation.
  \begin{align*}
    X^*\ltensor_R R/\ppp &= \rhom_R(X,R)\ltensor_R R/\ppp \\
    &\simeq \rhom_R(X,R/\ppp) \\
    &\simeq \rhom_R(X,\rhom_{R/\ppp}(R/\ppp,R/\ppp)) \\
    &\simeq \rhom_{R/\ppp}(X\ltensor_R R/\ppp,R/\ppp).
  \end{align*}
  Here, the first isomorphism follows from (Tensor-eval); the
  second is trivial; and the third is due to (Adjoint). To
  keep notation simple, let 
  \begin{equation*}
    (-)^{*_{R/\ppp}} = \rhom_{R/\ppp}(-,R/\ppp).
  \end{equation*} 
  We are required to demonstrate that
  \begin{equation*}
    \chi(X^*\ltensor_R R/\ppp) = (-1)^{\dim R/\ppp}\chi(X\ltensor_R R/\ppp),
  \end{equation*}
and since the Euler characteristics $\chi^R$ and $\chi^{R/\ppp}$ are
identical on all finite $R/\ppp$--complexes with finite length homology,
the computations above imply that we need to demonstrate that
\[
    \chi^{R/\ppp}((X\ltensor_R R/\ppp)^{*_{R/\ppp}}) =
    (-1)^{\dim R/\ppp}\chi^{R/\ppp}(X\ltensor_R R/\ppp).
\]
Having changed rings
  from $R$ to the domain $R/\ppp$, we need to verify that the element 
  $[X\ltensor_R R/\ppp]_{\proj(\mmm/\ppp)}$  in the Grothendieck space
    $\grot\proj(\mmm/\ppp)$ over $R/\ppp$ satisfies vanishing. But
    this follows from  the fact that $\alpha=[X]_{\proj(\XXX)}$ satisfies vanishing, since
  \begin{equation*}
   \chi^{R/\ppp}((X\ltensor_R R/\ppp)\ltensor_{R/\ppp}R/\aaa) =
   \chi^R(X\ltensor_R R/\aaa)= 0.
  \end{equation*}
  for all ideals $\aaa \in \var(\ppp)$ with $\dim R/\aaa<\dim R/\ppp=\codim\XXX$.
  Thus, it suffices to show that
  \begin{equation*}
   \chi(X^*)=(-1)^{\dim R}\chi(X)
  \end{equation*}
  when $R$ is a domain, $\XXX$ equals $\{\mmm\}$ and
  $[X]_{\proj(\mmm)}$ satisfies vanishing.

  \medskip
  \noindent
  $\mathbf{2^{\circ}}$ Without loss of generality, we may assume that $R$ is
  complete; in particular, we may assume that $R$ admits a normalized
  dualizing complex $D$. Letting $Y = R$ in~\eqref{chixiformulas}
  and applying Proposition~\ref{grothendieckobs}, it follows that
  \begin{equation*}
   \chi(X^*)=\chi(X^*,R)=\chi(X, D) = \chi(X, \H(D)).
  \end{equation*}
  According to~\ref{dualizing} we may assume that the modules
  in the dualizing complex $D$ have the form
  \begin{equation}\label{dualmod}
    D_i =  \!\!\!\!\!\bigoplus_{\dim R/\ppp = i} \!\!\!\!\! E_R(R/\ppp).
  \end{equation}
  Let
  $d=\dim R$ and observe that, since $[X]_{\proj(\mmm)}$ satisfies vanishing and
  \begin{equation*}
    \dim \H_i(D)\leq \dim D_i < d \quad\text{for all}\quad i<d,
  \end{equation*} 
  it follows that
  \begin{equation*}
   \chi^R(X^*) =\chi^R(X, \shift^{d} \H_{d}(D))
               = (-1)^{d} \chi^R(X, \H_{d}(D)).
  \end{equation*}
  Since $\H_d(D)$ is a submodule of $D_d$, there is a short exact sequence
  \begin{equation}\label{seq}
    0\to \H_{d}(D)\to D_d\to Q\to 0,
  \end{equation}
  where $Q$ is a submodule of $D_{d-1}$, so that $\dim Q\leq\dim D_{d-1}\leq d-1$,
  where the last inequality follows from \eqref{dualmod}. Since $R$ is
  assumed to be a domain,
\[
D_d=E(R)=R_{(0)},
\]
so localizing the short exact sequence \eqref{seq} at the prime ideal $(0)$, we obtain an 
  isomorphism
  \begin{equation*}
    \H_{d}(D)_{(0)}\stackrel{\cong}{\longrightarrow} R_{(0)}.
  \end{equation*}
This lifts to an $R$--homomorphism, producing an exact sequence of finitely
  generated $R$--modules
  \begin{equation*}
    0 \to K \to \H_{d}(D) \to R \to C \to 0,
  \end{equation*}
  where $K$ and $C$ are not supported at the prime ideal $(0)$. Thus,
  $\dim K$ and $\dim C$ are strictly smaller than $\dim R$. 
Consequently, since $[X]_{\proj(\XXX)}$ satisfies vanishing and the
 intersection multiplicity is additive on short exact sequences,
  \begin{align*}
   \chi(X^*) &= (-1)^d\chi(X, \H_d(D))\\ 
             &= (-1)^d(\chi(X, K) + \chi(X, R) -
                         \chi(X, C))\\ 
             &= (-1)^d\chi(X),
  \end{align*}
  which concludes the argument.

\medskip
\noindent
$\mathbf{3^{\circ}}$ 
We have now shown that, if $R$ satisfies vanishing, then $R$ satisfies
self-duality. To see the other implication, assume that $R$ satisfies
self-duality and let $\alpha$ be an element of $\grot\proj(\XXX)$ for
some specialization-closed subset $\XXX$ of $\spec R$. For any specialization-closed subset $\XXX'$
of $\spec R$ with $\XXX\subseteq\XXX'$ and $\codim\XXX'=\codim\XXX-1$,
we now have, for the image $\overline\alpha$ of $\alpha$ in $\grot\proj(\XXX')$,  that 
\[
(-1)^{\codim\XXX'}\overline\alpha^* =
\overline\alpha=\overline{(-1)^{\codim\XXX}\alpha^*}=
(-1)^{\codim\XXX}\overline\alpha^* ,
\]
which means that $\overline\alpha=0$. Thus, by
Proposition~\ref{injvandim0}, $\alpha$ satisfies vanishing, and
since $\alpha$ was arbitrary, $R$ must satisfy vanishing. Considering instead the image $\overline\alpha$
of $\alpha$ in $\grot\der(\XXX')$ and applying
Proposition~\ref{vanconditions}, the same argument shows that, if 
$\alpha$ is numerically self-dual, then $\alpha$ satisfies weak
vanishing. Thus, if $R$ satisfies numerical self-duality, then $R$
satisfies weak vanishing.
\end{proof}

\begin{theo}\label{dualitytheo}
  Assume that $R$ is complete of prime characteristic $p$ and with
  perfect residue field.   Let $\XXX$ be a specialization-closed
  subset of $\spec R$, let $\alpha\in\grot\proj(\XXX)$ and let
  $\beta\in\grot\inj(\XXX)$. Then, for all non-negative integers $i$  
  \begin{equation} \label{decostarformula}
   (\alpha^*)^{(i)}=(-1)^{i+\codim\XXX}\alpha^{(i)}
   \quad\text{and}\quad
   (\beta^\star)^{(i)}=(-1)^{i+\codim\XXX}\beta^{(i)} .
  \end{equation}
Consequently, if $u$ is the vanishing dimension of $\alpha$, then
  \begin{equation*}
    (-1)^{\codim\XXX}\alpha^*=\alpha^{(0)}-\alpha^{(1)}+\alpha^{(2)}-\cdots
    +(-1)^u\alpha^{(u)},
  \end{equation*}
and if $v$ is the vanishing dimension of $\beta$, then
  \begin{equation*}
    (-1)^{\codim\XXX}\beta^\star = \beta^{(0)}-\beta^{(1)}+\beta^{(2)}-\cdots
    +(-1)^v\beta^{(v)}.
  \end{equation*}
\end{theo}

\begin{proof}
The last two statements of the proposition are immediate consequences
of~\eqref{decostarformula}. We shall prove the formula for $\alpha$
in~\eqref{decostarformula}; the proof of the formula for $\beta$ is similar.

The proof is by induction on $i$. For $i = 0$, since $\alpha^{(0)}$
satisfies vanishing, the statement
  follows from Propositions~\ref{D3} and~\ref{selfdualityforvanishing}, since
  \begin{equation*}
   (\alpha^*)^{(0)}= (\alpha^{(0)})^* = (-1)^{\codim\XXX} \alpha^{(0)}.
  \end{equation*}
  Next, assume that $i>0$ and that the statement holds for smaller
  values of $i$. Choose an arbitrary specialization-closed subset $\XXX'$ of $\spec R$
  such that $\XXX\subseteq \XXX'$ and $\codim \XXX'=\codim \XXX-1$, and
  consider the element
  \begin{equation*}
   \sigma=(\alpha^*)^{(i)}-(-1)^{\codim\XXX+i}\alpha^{(i)}.
  \end{equation*}
We want to show that $\sigma=0$. Applying the automorphism
$\Phi_\XXX$, we get by Proposition~\ref{D3} that
  \begin{align*}
    \Phi_\XXX(\sigma) 
    &= \Phi_{\XXX}((\alpha^*)^{(i)})-(-1)^{\codim\XXX+i}\Phi_\XXX(\alpha^{(i)}) \\
    &= p^{-i}((\alpha^*)^{(i)}-(-1)^{\codim\XXX+i}\alpha^{(i)}) =
    p^{-i}\sigma ,
  \end{align*}
  showing that $\sigma$ is an eigenvector for $\Phi_\XXX$ with eigenvalue $p^{-i}$; in
  particular, we have $\sigma = \sigma^{(i)}$.  Denote by $\bar\sigma$
  the image of $\sigma$ in $\grot\proj(\XXX')$.  Then,
  by~\cite[Remark~20]{halvorsenF} (which corresponds to
  Remark~\ref{injdecompositioninclusion} but for elements of
  $\grot(\XXX)$) and the induction hypothesis we
  obtain 
  \begin{align*}
    \overline\sigma
    &= \overline{(\alpha^*)^{(i)}}-(-1)^{\codim\XXX+i}\overline{\alpha^{(i)}} \\
    &=
    {\overline{\alpha}^*}^{(i-1)}-(-1)^{\codim\XXX'+(i-1)}\overline\alpha^{(i-1)}
    = 0 .
  \end{align*}
  Consequently, by~\cite[Proposition~23]{halvorsenF} (which
  corresponds to Proposition~\ref{injvandim0} but for elements of $\grot\proj(\XXX)$), $\sigma$
  must satisfy vanishing: that is, $\sigma=\sigma^{(0)}$. But then
  $\sigma^{(i)} = \sigma = \sigma^{(0)}$ forcing $\sigma=0$.
\end{proof}

\begin{rema} \label{selfdualdecomposition}
  Assume that $R$ is complete of prime characteristic $p$ and with
  perfect residue field.  Let $\XXX$ be a specialization-closed subset of $\spec R$ and consider an element
  $\alpha\in\grot\proj(\XXX)$. In view of
  Theorem~\ref{diagonalizingthefrobenius} we may decompose $\alpha$
  into eigenvectors
  \begin{equation*}
     \alpha =\alpha^{(0)} + \alpha^{(1)} + \alpha^{(2)} +\cdots +\alpha^{(u)}
  \end{equation*}
  where $u$ is the vanishing dimension of $\alpha$. Comparing it with the
  decomposition of $\alpha^*$ from Theorem~\ref{dualitytheo}
  \begin{equation*}
     (-1)^{\codim\XXX}\alpha^*=\alpha^{(0)}-\alpha^{(1)}+\alpha^{(2)}-\cdots
     +(-1)^u\alpha^{(u)}
  \end{equation*}
  shows that $\alpha$ is self-dual if and only if $\alpha^{(i)}=0$ in
  $\grot\proj(\XXX)$ for all odd $i$: that is, if
  and only if
\[
\alpha=\alpha^{(0)}+\alpha^{(2)}+\cdots
\]
in $\grot\proj(\XXX)$. Similarly, $\alpha$ is numerically self-dual if
and only if $\overline{\alpha^{(i)}}=0$ in $\grot\der(\XXX)$ for all
odd $i$: that is, if and only if
\[
\overline\alpha=\overline{\alpha^{(0)}}+\overline{\alpha^{(2)}}+\cdots
\]
in $\grot\der(\XXX)$. Similar considerations apply for
elements $\beta\in\grot\inj(\XXX)$.
\end{rema}

In Proposition~\ref{selfdualityforvanishing}, we proved that vanishing
and self-duality are equivalent for $R$ and that numerical
self-duality implies weak vanishing. The following proposition shows
that, in characteristic $p$, numerical vanishing logically  lies
between self-duality and numerical self-duality.

\begin{prop}\label{Rconditions} 
Assume that $R$ is complete of prime characteristic $p$ and with
perfect residue field. For the following conditions, each condition implies the
  next.   In fact, \eqref{Rvan} and~\eqref{Rselfdual} are equivalent.
  \begin{enumerate}
    \item \label{Rvan} $R$ satisfies vanishing.
    \item \label{Rselfdual} $R$ satisfies self-duality.
    \item \label{Rnumvan} $R$ satisfies numerical vanishing.
    \item \label{Rnumselfdual} $R$ satisfies numerical self-duality.
    \item \label{Rweakvan} $R$ satisfies weak vanishing.
  \end{enumerate}
\end{prop}
\begin{proof}
 The equivalence of \eqref{Rvan} and \eqref{Rselfdual} and the fact that \eqref{Rnumselfdual} implies
 \eqref{Rweakvan} is contained in
 Proposition~\ref{selfdualityforvanishing}. The fact that \eqref{Rvan}
 implies \eqref{Rnumvan} is contained in~\cite[Proposition~27]{halvorsenF},
 and Remark~\ref{selfdualdecomposition} makes it clear that \eqref{Rnumvan}
 implies \eqref{Rnumselfdual}. 
\end{proof}

\begin{rema}
  The constructions by Miller and Singh~\cite{millersingh} shows that there
  can exist elements satisfying self-duality but not vanishing as well as
  elements satisfying numerical self-duality but not numerical vanishing;
  see~\cite[Example~35]{halvorsenF} for further details on this example.
  Roberts~\cite{robertsI} has shown the existence of a ring satisfying
  weak vanishing but not numerical self-duality;
  see~\cite[Example~32]{halvorsenF} for further details. Thus, all the
  implications except the equivalence in the preceding proposition are strict.
\end{rema}

\begin{prop} \label{chixivanishing}
  $R$ satisfies vanishing precisely when 
  \begin{equation}\label{tensorhom}
    \alpha\otimes\gamma = (-1)^{\codim \XXX}\hom(\alpha,\gamma)
  \end{equation}
  in $\grot\der(\mmm)$ for all specialization-closed subsets $\XXX$ of $\spec R$, all
  $\alpha\in\grot\proj(\XXX)$ and all $\gamma\in\grot\der(\comp\XXX)$,
  and $R$ satisfies numerical 
  self-duality  precisely when~\eqref{tensorhom} holds in
  $\grot\der(\mmm)$ when requiring $\gamma\in\grot\proj(\comp\XXX)$
  instead. In other words, $R$ satisfies vanishing 
  precisely when the intersection multiplicity and the Euler form 
  satisfy the identity
  \begin{equation}\label{chixi}
     \chi(X,Y)=(-1)^{\codim(\supp X)}\xi(X,Y)
  \end{equation}
  for all complexes $X\in\proj(R)$ and $Y\in\der(R)$ with
\[
\supp X\cap\supp Y=\{\mmm\}\quad\text{and}\quad \dim(\supp
X)+\dim(\supp Y)\leq \dim R,
\]
and $R$ satisfies numerical self-duality precisely
  when~\eqref{chixi} holds when restricting to
  complexes $Y\in\proj(R)$.
\end{prop}
\begin{proof}
  Employing Proposition~\ref{formulas} it is readily verified that~\eqref{tensorhom} is
  equivalent to 
\[
\alpha\otimes\gamma=(-1)^{\codim\XXX}\alpha^*\otimes\gamma.
\]
However, this identity is satisfied for all
$\gamma\in\grot\der(\comp\XXX)$ precisely when $\alpha$ is
self-dual. From Proposition~\ref{selfdualityforvanishing} it follows that $R$
satisfies vanishing if and only if~\eqref{tensorhom} holds for all specialization-closed subsets
  $\XXX$ of $\spec R$, all $\alpha\in\grot\proj(\XXX)$ and all
  $\gamma\in\grot\der(\comp\XXX)$.

  On the other hand, applying the above argument to the case where
  $\gamma\in\grot\proj(\comp\XXX)$ shows that $R$ satisfies numerical
  self-duality presily when~\eqref{tensorhom} is satisfied for all specialization-closed
  subsets $\XXX$ of $\spec R$, all $\alpha\in\grot\proj(\XXX)$ and
  all $\gamma\in\grot\proj(\comp\XXX)$.

  Assume next that~\eqref{tensorhom} holds in $\grot\der(\mmm)$ for
  all specialization-closed subsets $\XXX$ of $\spec R$, all
  $\alpha\in\grot\proj(\XXX)$ and all
  $\gamma\in\grot\der(\comp\XXX)$. If $X\in\proj(R)$  
  and $Y\in\der(R)$ are complexes such that 
\begin{equation}\label{frankild}
\supp X\cap\supp Y=\{\mmm\}\quad\text{and}\quad \dim(\supp
X)+\dim(\supp  Y)\leq \dim R,
\end{equation}
the identity~\eqref{chixi} follows by setting
\[
\XXX=\supp X,\quad \alpha=[X]_{\proj(\XXX)}\quad\text{and}\quad
\gamma=[Y]_{\der(\comp\XXX)}
\]
in~\eqref{tensorhom}.
Conversely, if \eqref{chixi} holds for all
 complexes $X\in\proj(R)$ and $Y\in\der(R)$ such that~\eqref{frankild}
  is satisfied, then~\eqref{tensorhom} follows for all specialization-closed 
  subsets $\XXX$ of $\spec R$, all $\alpha\in\grot\proj(\XXX)$ and all
  $\gamma\in\grot\der(\comp\XXX)$, since we by 
  Proposition~\ref{grothendieckobs},
  $\alpha=r[X]_{\proj(\XXX)}$ for an $r\in\QQ$ and a 
  complex $X\in\proj(\XXX)$ with $\codim(\supp X)=\codim
  \XXX$. Applying the same argument to elements
  $\gamma\in\grot\proj(\comp\XXX)$ and complexes $Y\in\proj(R)$
  proves the last part of the proposition.
\end{proof}

\begin{rema}
    Proposition~\ref{chixivanishing} confirms Chan's
  supposition in~\cite{chan}, in the setting of complexes rather than
  modules, that the formula in~\eqref{chixi} is equivalent to the
  vanishing conjecture.
Note that, when restricting attention to complexes $Y$ in $\proj(R)$,
formula~\eqref{chixi} is equivalent to numerical self-duality, which
implies the weak vanishing conjecture but need not be equivalent to
it. This negatively answers the question of whether the restriction of
the formula in~\eqref{chixi}  to complexes $Y$ in $\proj(R)$  is equivalent to the weak vanishing conjecture.
\end{rema}

We already know that, if $R$ is regular, then $R$ satisfies vanishing,
whereas, if $R$ is a complete intersection (which is complete of prime
characteristic $p$ and with perfect residue field), then $R$ satisfies numerical
vanishing; see~\cite[Example~33]{halvorsenF}. The authors believe that this line of implications can be
continued, at least in the characteristic $p$ case, with the claim that, if $R$ is Gorenstein, $R$ satisfies
numerical self-duality, so that we have the following implications of
properties of $R$ in the case where $R$ is complete of prime
characteristic $p$ and with perfect residue field.
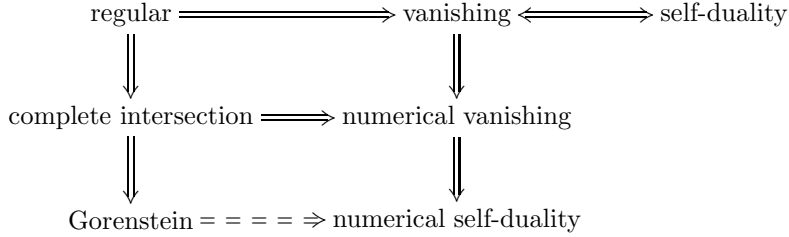
\[
\xymatrix{
{\text{regular}} \ar@{=>}[r] \ar@{=>}[d]  & {\text{vanishing}}
\ar@{=>}[d] & {\text{self-duality}} \ar@{<=>}[l]\\
{\text{complete intersection}} \ar@{=>}[r] \ar@{=>}[d]  &
{\text{numerical vanishing}} \ar@{=>}[d] \\
{\text{Gorenstein}}  \ar@{==>}[r]  & {\text{numerical self-duality}} \\
}
\]
This supposition complies with the following
proposition.

\begin{prop} \label{Rgorenstein}
  Assume that $R$ is Gorenstein and let $\XXX$ be a specialization-closed subset of
  $\spec R$. If $\dim\XXX\leq 2$, then all elements of
  $\grot\proj(\XXX)$ are numerically self-dual. In particular, if
  $\dim R\leq 5$, then $R$ satisfies numerical self-duality.
\end{prop}

\begin{proof}
  Let $\XXX$ be a specialization-closed subset of $\spec R$ with $\dim\XXX\leq 2$ and
  consider elements $\alpha$ in $\grot\proj(\XXX)$ and $\beta$ in
  $\grot\proj(\comp{\XXX})$. Then $\codim\comp{\XXX}\leq 2$, and
  therefore $\beta$ satisfies vanishing by Proposition~\ref{vanconditions};
  in particular,
  \begin{equation*}
    \beta^*=(-1)^{\codim\comp{\XXX}}\beta = (-1)^{\dim R-\codim\XXX}\beta.
  \end{equation*}
  When $R$ is Gorenstein, the complex $D=\shift^{\dim R}R$ is a
  normalized dualizing complex for $R$ forcing
  $(-)^\dagger=\shift^{\dim R}(-)^*$. Thus, applying
  Proposition~\ref{formulas} the identity
  \begin{equation*}
    \alpha^*\otimes\beta = \alpha\otimes\beta^\dagger 
    = (-1)^{\dim R}\alpha\otimes\beta^* 
    = (-1)^{\codim \XXX}\alpha\otimes\beta 
  \end{equation*}
  holds in $\grot\der(\mmm)$. This proves that
  $\alpha^*=(-1)^{\codim\XXX}\alpha$ so that $\alpha$ is numerically self-dual.
  
  If $\dim R\leq 5$ then any specialization-closed subset $\XXX$ of $\spec R$
  must either satisfy $\codim\XXX\leq 2$, in which case vanishing
  holds in $\grot\proj(\XXX)$ by Proposition~\ref{vanconditions}, or $\dim\XXX\leq 2$. In either case,
  all elements of $\grot\proj(\XXX)$ are numerically self-dual.
\end{proof}

Since numerical self-duality implies weak vanishing, the preceding
proposition shows that weak vanishing holds for any Gorenstein ring of
dimension at most $5$. Dutta~\cite{duttaG2} has already proven
this fact.

\providecommand{\MR}{\relax\ifhmode\unskip\space\fi MR }
\providecommand{\MRhref}[2]{%
  \href{http://www.ams.org/mathscinet-getitem?mr=#1}{#2}
}
\providecommand{\href}[2]{#2}

\end{document}